\pdfoutput=1
\documentclass[twoside,11pt,reqno]{amsart}
\usepackage{amsmath,amssymb,amscd,mathrsfs,amscd, todonotes}
\usepackage{graphics,verbatim}
\usepackage{todonotes}
\usepackage{stmaryrd}

\usepackage[breaklinks]{hyperref}
\let\oldsection=\section

\newcommand{\losemi}{{\otimes \kern -.78em \ltimes}}
\newcommand{\rosemi}{{\otimes \kern -.78em \rtimes}}

\newcommand{\Hom}{\ensuremath{\operatorname{Hom}}}

\newcommand{\Ind}{\ensuremath{\operatorname{ind}}}

\newcommand{\Ext}{\operatorname{Ext}}

\newcommand{\Z}{\mathbb{Z}}
\newcommand{\C}{\mathbb{C}}

\newcommand{\res}{\ensuremath{\operatorname{res}}}
\newcommand{\rk}{\operatorname{rank}}
\newcommand{\X}{\mathcal{X}}

\newcommand{\fg}{\ensuremath{\mathfrak{g}}}
\newcommand{\fb}{\ensuremath{\mathfrak{b}}}
\newcommand{\fu}{\ensuremath{\mathfrak{u}}}
\newcommand{\fh}{\ensuremath{\mathfrak{h}}}

\newcommand{\fp}{\ensuremath{\mathfrak{p}}}
\newcommand{\ft}{\ensuremath{\mathfrak{h}}}

\renewcommand{\d}{\delta}

\newcommand{\ga}{\gamma}

\newcommand{\V}{\mathcal{V}}
\newcommand{\HH}{\operatorname{H}}

\newcommand{\R}{\ensuremath{\mathbb{R}}}

\newcommand{\CC}{\mathcal{C}}

\newcommand{\XX}{\mathcal{X}}

\newcommand{\Tensor}{\operatorname{Tensor}}
\newcommand{\Loc}{\operatorname{Loc}}
\newcommand{\Proj}{\operatorname{Proj}}
\newcommand{\stmod}{\operatorname{stmod}}

\newcommand{\rank}{\operatorname{rank}}
\newcommand{\bT}{\mathbf T}

\newcommand{\bC}{\mathbf C}
\newcommand{\bK}{\mathbf K}

\newcommand{\bP}{\mathbf P}
\newcommand{\bI}{\mathbf I}
\newcommand{\bR}{\mathbf R}

\newcommand{\unit}{\ensuremath{\mathbf 1}}

\renewcommand{\AA}{\mathcal{A}}
\newcommand{\NN}{\mathcal{N}}
\newcommand{\UU}{\mathcal{U}}
\newcommand{\BU}{\mathbb{U}}
\renewcommand{\mod}{\operatorname{mod}}
\newcommand{\Mod}{\operatorname{Mod}}
\newcommand{\OO}{\mathcal{O}}

\newcommand{\gr}{\operatorname{gr}}
\newcommand{\Spec}{\operatorname{Spec}}
\newcommand{\Stab}{\operatorname{Stab}}
\newcommand{\Stmod}{\operatorname{Stmod}}

\newcommand{\cl}{\operatorname{cl}}
\newcommand{\ZZ}{Z}
\newcommand{\supp}{\operatorname{supp}}

\newcommand{\gf}[2]{\genfrac{}{}{0pt}{}{#1}{#2}}
\renewcommand{\Gamma}{\varGamma}

\makeatletter
\newcommand{\leqnomode}{\tagsleft@true}
\newcommand{\reqnomode}{\tagsleft@false}
\makeatother

\newtheorem{theorem}{Theorem}[subsection]

\makeatletter\let\c@fact\c@theorem\makeatother

\makeatletter\let\c@note\c@theorem\makeatother

\newtheorem{lemma}{Lemma}[subsection]
\makeatletter\let\c@lemma\c@theorem\makeatother

\newtheorem{assumption}{Assumption}[subsection]
\makeatletter\let\c@lemma\c@theorem\makeatother

\makeatletter\let\c@alg\c@theorem\makeatother

\makeatletter\let\c@remark\c@theorem\makeatother

\newtheorem{prop}{Proposition}[subsection]
\makeatletter\let\c@prop\c@theorem\makeatother

\makeatletter\let\c@conj\c@theorem\makeatother

\newtheorem{cor}{Corollary}[subsection]
\makeatletter\let\c@cor\c@theorem\makeatother

\newtheorem{defn}{Definition}[subsection]
\makeatletter\let\c@defn\c@theorem\makeatother

\theoremstyle{definition}

\makeatletter\let\c@example\c@theorem\makeatother
\numberwithin{equation}{subsection}

\usepackage[capitalise]{cleveref}
\crefname{theorem}{Theorem}{Theorems}
\crefname{fact}{Fact}{Facts}
\crefname{note}{Note}{Notes}
\crefname{lemma}{Lemma}{Lemmas}
\crefname{alg}{Algorithm}{Algorithms}
\crefname{remark}{Remark}{Remarks}
\crefname{example}{Example}{Examples}
\crefname{prop}{Proposition}{Propositions}
\crefname{conj}{Conjecture}{Conjectures}
\crefname{cor}{Corollary}{Corollaries}
\crefname{defn}{Definition}{Definitions}
\crefname{equation}{\!\!}{\!\!} %Remove spacing around phantom equation name

\newcounter{listequation}
\newenvironment{eqnlist}{\begin{list}
{(\thesubsection.\thelistequation)}
{\usecounter{listequation} \setlength{\itemsep}{1.0ex plus.2ex
minus.1ex}}\setcounter{listequation}{\value{equation}}}
{\setcounter{equation}{\value{listequation}}\end{list}}

\newcounter{mytempcount} % for the two-part enumeration of quantum group notation

\begin{document}
\title{Tensor Triangular Geometry for Quantum Groups}

\author{Brian D. Boe }
\address{Department of Mathematics \\
          University of Georgia \\
          Athens, GA 30602}
\email{brian@math.uga.edu}
\author{Jonathan R. Kujawa}
\address{Department of Mathematics \\
          University of Oklahoma \\
          Norman, OK 73019}
\thanks{Research of the second author was partially supported by NSF grant
DMS-1160763 and NSA grant H98230-16-0055}\
\email{kujawa@math.ou.edu}
\author{Daniel K. Nakano}
\address{Department of Mathematics \\
          University of Georgia \\
          Athens, GA 30602}
\thanks{Research of the third author was partially supported by NSF
grant  DMS-1402271 and DMS-1701768}\
\email{nakano@math.uga.edu}
\date{\today}

\begin{abstract} Let $\fg$ be a complex simple Lie algebra and let $U_{\zeta}(\fg)$ be the corresponding Lusztig $\Z[q,q^{-1}]$-form of the quantized enveloping algebra specialized to an $\ell$th root of unity. Moreover, let $\mod(U_{\zeta}(\fg))$ be the braided monoidal category of finite-dimensional modules for $U_{\zeta}(\fg)$. 
In this paper we classify the thick tensor ideals of $\mod(U_{\zeta}(\fg))$ and compute the prime spectrum of the stable module category associated to 
$\mod(U_{\zeta}(\fg))$ as defined by Balmer. 
\end{abstract}

\maketitle

\section{Introduction}\label{S:intro}   

\subsection{} The geometry of the nilpotent cone $\NN$ plays a central role in the seminal work of Arkhipov, Bezrukavnikov, and Ginzburg \cite{ABG} on the representation theory for quantum groups at a root of unity. In particular, they showed that there are derived equivalences between the principal block 
for the large quantum group $U_{\zeta}(\fg)$, equivariant coherent sheaves on the Springer resolution of $\NN$, and perverse sheaves on the loop Grassmannian. One consequence of these equivalences of triangulated categories is a proof of Lusztig's character formula for quantum groups when $\ell>h$ (where $\ell$ is the order of the root of unity and $h$ is the Coxeter number for $\fg$).  An important underlying idea in \cite{ABG} is the lifting of the support variety theory for the small quantum group to obtain the equivalence between the principal block and coherent sheaves on the Springer resolution. This approach illustrates that to calculate characters of simple modules at the representation theoretic level one needs to examine the structure of the underlying tensor triangulated category at the derived level where the role of geometry is more transparent.

As further evidence of this connection, Drupieski, Nakano, and Parshall \cite{DNP}, using the validity of the Lusztig character formula and the positivity of the coefficients of the Kazhdan-Lusztig polynomials for the affine Weyl group, calculated the support varieties for all irreducible representations for the small quantum group at an $\ell$th root of unity.  Lusztig demonstrated that there is a bijection between two-sided cells of the affine Weyl group and nilpotent orbits in $\NN$. Humphreys conjectured that the support variety of a tilting module for $U_{\zeta}(\fg)$, when restricted to the small quantum group, would be given by the nilpotent orbit whose corresponding two-sided cell contains the high weight of the tilting module. Bezrukavnikov \cite{Bez} verified a version of this conjecture using the tools from \cite{ABG} described above. Bezrukavnikov's result involves so-called ``relative support varieties.'' In the case of tilting modules with regular highest weight, we verify Humphreys' conjecture  for ``absolute support varieties'' in \cref{P:orbitrealization}.

Let $\bT$ be the category of finite-dimensional tilting modules for $U_{\zeta}(\fg)$. This category is a braided monoidal category via the coproduct on the Hopf algebra $U_{\zeta}(\fg)$. Ostrik \cite{Ost} showed that the thick tensor ideals in $\bT$ are in one-to-one correspondence with two-sided cells for the affine Weyl group, or equivalently (via Lusztig's bijection), nilpotent orbits in $\NN$. In particular, two tilting modules generate the same thick tensor ideal in $\bT$ if and only if they have the same support varieties. In contrast,  the classification of the thick tensor ideals in the category of all finite-dimensional $U_{\zeta}(\fg )$-modules has proven elusive. 

\subsection{} Given a braided monoidal category, a fundamental question is to classify its thick tensor ideals (see \cref{SS:TTC} for definitions). In the same setting, Balmer showed that one can extract the ambient geometry of the category by using the tensor structure to treat the category like a commutative ring. He defined a notion of prime tensor ideals in the category and used this to construct a categorical spectrum (referred to in this paper as the Balmer spectrum). Determining the Balmer spectrum  is thus important because it reveals a natural geometric object that governs the structure and interaction of objects in the category.  Furthermore, Balmer proved that determining the Balmer spectrum is intimately related to the classification of the thick tensor ideals.  An example closely related to the results of this paper is that the Balmer spectrum for restricted Lie algebras (for the stable module category) identifies with the restricted nilpotent cone \cite{Bal}. Another example, in the same spirit, is in an earlier paper of the authors which determined the Balmer spectrum for classical Lie superalgebras by utilizing their detecting subalgebras \cite{BKN}. 

The main goals of this paper are to (i) classify the thick tensor ideals in $\mod(U_{\zeta}(\fg))$ and (ii) determine the Balmer spectrum for the stable module category associated to $\mod(U_{\zeta}(\fg))$ (see \cref{SS:modulecategories} for definitions).  One can view our results as complementary to Ostrik's classification of tensor ideals in $\bT$.   Both (i) and (ii) will involve much of the machinery developed by the authors in \cite{BKN}.  In particular, this entails constructing a support data from the stable module category to an appropriate Zariski space. The geometric object of central importance turns out to be the collection of $G$-invariant ideals in the coordinate algebra of $\NN$. Furthermore, we rely heavily on the use of infinitely generated modules for $U_{\zeta}(\fg)$ and localization theorems in a stable module category which contains these infinitely generated modules.

A fundamental obstacle we need to overcome is to establish various versions of the tensor product property for the support theories for modules over the large and small quantum groups. 
This is addressed as follows. First, we define the notion of quasi support data whose axioms involve a weaker form of the tensor product condition. Second, to obtain an analogue of a result of Hopkins, a key assumption (i.e., 
Assumption~\ref{A:projectivity}) is identified that  involves the projectivity of infinitely generated modules.  We then reduce, via line bundle cohomology techniques, this issue of projectivity to modules for the small quantum group attached to the Borel subalgebra $\fb \subseteq \fg$. The next step is to pass to a designated associated graded algebra, $\gr u_{\zeta}(\fb)$, for $u_{\zeta}(\fb)$. It is in this category that we prove a tensor product theorem (cf.\ Theorem~\ref{T:grWintersection}). 
The main idea of the proof is to apply work of Benson, Erdmann, and Holloway \cite{BEH} to relate supports of modules in $\gr u_{\zeta}(\fb)$ to supports for a quantum complete intersection (with equal parameters) where the theory is better understood via rank varieties. By carefully keeping track of the relationship between the support theories between $u_{\zeta}(\fb)$ and $\gr u_{\zeta}(\fb)$ we can successfully verify the key assumption. Much of our analysis uses key facts in the  local support theory as developed by Benson, Iyengar, and Krause \cite{BIK}. 

In Proposition~\ref{P:TensorProductProperty} and Theorem~\ref{T:classification}, as an application of the classification of thick tensor ideals for $\stmod(U_{\zeta}(\fg))$, we prove the tensor product property for supports of $U_{\zeta}(\fg)$-modules. 
The tensor product property is then applied in the construction of a homeomorphism to compute the Balmer spectrum of $\stmod(U_{\zeta}(\fg))$ (cf.\ Theorem~\ref{T:spcquantum}). The problem of proving the tensor product property for modules over the small quantum group $u_{\zeta}(\fg)$ remains open.  

\subsection{Connections to Affine Lie Algebras} Our results have analogues for affine Lie algebras which we now describe. Let $\fg$ be a finite-dimensional complex simple Lie algebra, $\hat{\fg} = \C [t,t^{-1}] \otimes \fg \oplus \C c \oplus\C d$ be the corresponding untwisted affine Lie algebra, and $\tilde{\fg}$ be the subalgebra $\tilde{\fg}=\C [t,t^{-1}] \otimes \fg \oplus \C c \subseteq \hat{\fg}$.  For $\kappa \in \C$ we let $\OO_{\kappa}$ be the full subcategory of all $\tilde{\fg}$-modules, $M$, for which the central element $c$ acts by $\kappa$ and $M$ satisfies certain category $\mathcal{O}$-type finiteness conditions.  

 Now we set $\kappa=(-\ell)/2D-h^{\vee}$, where $h^{\vee}$ is the dual Coxeter number for $\fg$ and 
 $D$ is the maximal number of edges between two nodes in the Dynkin diagram of $\fg$.
The category $\OO_{\kappa}$ admits a tensor product and duality which makes it into a braided monoidal category which is rigid (i.e.,\ with duals).
 Under mild assumptions on $\ell$, Kazhdan-Lusztig \cite{KL, KL2, KL3} and Lusztig \cite{Lusztig} prove in the simply-laced and non-simply-laced cases, respectively, that there exists a functor 
\[
F_{\ell}: \OO_{\kappa} \to \mod(U_{\zeta}(\fg))
\] which is an equivalence of braided monoidal categories.  See \cite{Tanisaki} or \cite[Section 16.3]{CP} for an overview of these results.  From this it immediately follows that the classification of thick tensor ideals in the two categories coincides and that the 
Balmer spectra for the corresponding stable module categories are homeomorphic.  That is, the appropriate analogues of Theorems~\ref{T:classification}~and~\ref{T:spcquantum} also hold for $\OO_{\kappa}$.

\subsection{Acknowledgements} The second and third authors are pleased to acknowledge the hospitality and support of the Mittag-Leffler Institute during the special semester in Representation Theory in Spring 2015. The formulations of several of the 
key ideas in the paper were discovered at this time. The authors thank Cris Negron for spotting a gap in an earlier version of Theorem~\ref{T:naturality}, and for his insights into the questions of naturality of supports for quantum groups.

\section{Preliminaries} \label{S:preliminaries}
 
\subsection{Tensor Triangulated Categories}\label{SS:TTC} For the purposes of this paper we will summarize the notion of tensor triangulated categories as defined by Balmer (cf.\ \cite[Definition 1.1]{Bal}).  We will not state all the definitions but instead refer the reader to \cite{BKN} for a full treatment. The main idea is to use the tensor structure as a ``multiplication'' and define categorical analogues of prime ideals and the spectrum. 

\begin{defn} \label{TTCdefinition} A {\em tensor triangulated category} (TTC) is a triple $(\bK, \otimes, \unit)$ such that 
\begin{itemize} 
\item[(i)] $\bK$ is a triangulated category, and 
\item[(ii)] $\bK$ has a braided\footnote{To be precise, Balmer assumes the tensor product is symmetric but his definitions and results apply equally well in the braided setting.} monoidal tensor product $\otimes: \bK \times \bK \rightarrow \bK$ which is exact in each variable with unit object $\unit$.  
\end{itemize} 
\end{defn} 

A \emph{(tensor) ideal} in $\bK$ is a triangulated subcategory $\bI$ of $\bK$ such that $M\otimes N\in \bI$ for all $M\in \bI$ and $N\in \bK$, and an ideal is called \emph{thick} if it is closed under the taking of direct summands. A \emph{prime ideal}, $\bP$, of $\bK$ is a proper thick tensor ideal such that if $M\otimes N\in \bP$, then either $M\in \bP$ or $N\in \bP$ (see \cite[Definition 2.1]{Bal}). 

The \emph{Balmer spectrum} of $\bK$ is defined to be 
$$\operatorname{Spc}(\bK)=\{\bP \subset \bK \mid \bP\ \text{is a prime ideal}\}.$$
The topology on $\operatorname{Spc}({\bK})$ is given by closed sets of the form 
$${\mathcal Z}({\mathcal C})=\{\bP\in \operatorname{Spc}(\bK) \mid {\mathcal C}\cap \bP=\varnothing \}$$ 
where ${\mathcal C}$ is a family of objects in $\bK$. 

For a given TTC, $\bK$, let $\bK^{c}$ denote the full triangulated subcategory of compact objects.  When we say $\bK$ is a compactly generated TTC we mean that $\bK$ is closed under arbitrary set indexed coproducts, the tensor product preserves set indexed coproducts, $\bK$ is compactly generated, the tensor product of compact objects is compact, $\unit$ is a compact object, and every compact object is rigid (i.e.\  strongly dualizable) as in, for example, \cite{HPS}.   In particular we have an exact contravariant duality functor $(-)^{*}:\bK^{c} \rightarrow \bK^{c}$ such that 
$$\Hom_{{\bK}}(M\otimes N, Q)=\Hom_{\bK}(N,M^{*} \otimes Q)$$ 
for $M\in \bK^{c}$ and $N,Q\in \bK$. We refer the reader to \cite[Section 2.1]{BKN} for further discussion about compactness.

\subsection{Zariski Spaces}\label{SS:zariski}  A Noetherian topological space 
$X$  is a \emph{Zariski space} if, in addition, any irreducible closed set $Y$ of $X$ has a unique generic point (i.e., there exists a unique $y\in Y$ such that $Y=\overline{\{y\}}$). 
Note that for a Noetherian topological space $X$ any closed set in $X$ is the union of finitely many irreducible closed sets. In particular we 
will be primarily interested in the cases when $X=\Spec(R)$ where $R$ is a commutative Noetherian ring, or when 
$R$ is graded, $X=\Proj(R)$. We set the following notation for a Zariski space $X$: 
\begin{itemize} 
\item[(i)] $\XX$: the collection of all subsets of $X$, 
\item[(ii)] $\XX_{sp}$: the collection of all specialization closed subsets (i.e., sets which are arbitrary unions of closed sets),
\item[(iii)] $\XX_{cl}$: the collection of all closed subsets,  
\item[(iv)] $\XX_{irr}$: the collection of all irreducible closed subsets. 
\end{itemize}

We will be interested in cases where we have an algebraic group $G$ acting rationally on a graded commutative ring $R$ by automorphisms which preserve the grading. This action induces an action of $G$ on $X=\Proj( R)$.  Let $X_{G}=G\text{-}\Proj( R)$ be the set of homogeneous $G$-prime ideals of $R$  (e.g., as defined in \cite{Lorenz2009}). The 
topology in $X_{G}$ is induced from the map $\rho:X\twoheadrightarrow X_{G}$ with $\rho( P)=\cap_{g\in G}\ gP=:\cap_{g}\ gP$ by declaring $W\subseteq X_{G}$ closed if and only if $\rho^{-1}(W)$ is closed in $X$. We have $\cap_{g}\ gP_{1}=\cap_{g}\ gP_{2}$ for $P_{1}, P_{2}\in X$ if and only if $\overline{G\cdot P_{1}}=\overline{G\cdot P_{2}}$ in $X$. 
In \cite[Section 2.3]{BKN} it was shown that $X_{G}$ is a Zariski space. 

\subsection{Localization functors}\label{SS:localization} In this section we briefly recall a key tool: the localization and colocalization 
functors as given in \cite[Proposition 2.16]{BIK:12}.  For details, the reader is referred to \cite[Section 2.2]{BIK:12}, and the general properties given in \cite{BIK}. 

Assume that the triangulated category $\bK$ admits arbitrary set indexed coproducts. 
A \emph{localizing subcategory} of $\bK$ is a  triangulated subcategory which is closed under taking set indexed coproducts. 
For $\CC$ a collection of objects of $\bK$, let $\Loc (\CC )$ 
be the smallest localizing subcategory containing $\CC$. 

The result we need is the following restatement of \cite[Theorem 2.32]{BIK:12}.  

\begin{theorem} \label{T:localizationtriangles} Let $\bK$ be a compactly generated triangulated category which admits arbitrary set indexed coproducts.  Given a thick subcategory $\bC$ of $\bK^{c}$ and an object 
$M$ in $\bK$, there exists a functorial triangle in $\bK$,
$$
\Gamma_{\bC}(M) \to M \to L_{\bC}(M) \to
$$
which is unique up to isomorphism, such that $\Gamma_{\bC}(M)$ is in $\Loc(\bC)$ and there are no non-zero maps in $\bK$ from $\bC$ or, equivalently, from $\Loc(\bC)$ to $L_{\bC}(M)$.
\end{theorem}

\subsection{Quasi Support Data} \label{SS:supportdata} Balmer \cite{Bal} provided a definition of support data as a method to relate objects in a TTC to  subsets in a given Zariski space. 
For our purposes we will need a weaker notion which relaxes the condition on the tensor products of objects. We will call this notion 
a quasi support data. 

Let $\bK$ be a TTC (possibly $\bK=\bK^{c}$), and $X$ be a Zariski space. 
A {\em quasi support data} is an assignment $V:\bK\rightarrow \XX$ which satisfies the following six properties (for $M, M_{i}, N, Q \in \bK$): 
\begin{eqnlist}
\item \label{E:supportone} $V(0)=\varnothing$, $V(\unit)=X$;
\item \label{E:supporttwo} $V( \oplus_{i \in I} M_{i})=\bigcup_{i \in I} V(M_{i})$ whenever $ \oplus_{i \in I} M_{i}$ is an object of $\bK$;
\item \label{E:supportthree} $V(\Sigma M)=V(M)$, where $\Sigma$ is the shift functor for $\bK$;
\item \label{E:supportfour} for any distinguished triangle $M\rightarrow N \rightarrow Q \rightarrow \Sigma M$ we have 
$$V(N)\subseteq V(M)\cup V(Q);$$ 
\item \label{E:supportfive} $V(M\otimes N)\subseteq V(M)\cap V(N)$;
%\item \label{E:supportsix} $V(M)=V(M^{*})$ for $M\in\bK^{c}$.
\end{eqnlist} 
We note that using the above properties and \cite[Lemma A.2.6]{HPS} one can verify, when $\bK$ is a compactly generated TTC, that 
\leqnomode
\begin{align}\label{E:supportsix} 
V(M)=V(M^{*}) \text{ for } M\in\bK^{c}.
\end{align}
\reqnomode

We will be interested in quasi support data which satisfy one or both of two additional properties: 
\begin{eqnlist} 
\item \label{E:supportseven} $V(M)=\varnothing$ if and only if $M=0$;
\item \label{E:supporteight} for any $W\in \XX_{cl}$ there is an $M\in \bK^{c}$ such that $V(M)=W$ (Realization Property). 
\end{eqnlist} 

\subsection{An Additional Assumption on a TTC} Let $V:\bK^{c}\rightarrow \XX_{cl}$ be a quasi support data, as above. Given an object $M\in\bK^{c}$, let $\Tensor(M)\subseteq \bK^{c}$ be the thick tensor ideal in $\bK^{c}$ generated by $M$. 
For $W\in {\mathcal X}$, set 
\begin{equation} \label{E:IWdefinition}
\bI_{W}=\{Q\in \bK^{c} \mid V(Q)\subseteq W \}.
\end{equation}
By (\ref{SS:supportdata}.\ref{E:supporttwo})--(\ref{SS:supportdata}.\ref{E:supportfive}), $\bI_{W}$ is a thick tensor ideal of $\bK^{c}$. 
We state the following key assumption on $\bK$ (which will be verified later for quantum groups). 

\begin{assumption} \label{A:projectivity} Suppose $M\in \bK^{c}$ with $M\neq 0$, $N\in \bI_{V(M)}\subseteq \bK^{c}$, and $\bI'=\Tensor(M)$. If 
$M\otimes L_{\bI'} (N)=0$, then $L_{\bI'} (N)=0$. 
\end{assumption} 

In many circumstances (e.g., finite group schemes, Lie superalgebras) Assumption~\ref{A:projectivity} is verified by showing the 
existence of an extension of support data from $\bK^{c}$ to $\bK$ which satisfies the tensor product property (see \cite[Section 3.2]{BKN} for a discussion of extending support data). For quantum groups, suitably extending the quasi support data to non-compact objects is an open problem.  In this situation alternative methods will be necessary to verify Assumption~\ref{A:projectivity}. 

\subsection{Hopkins' Theorem}  We will now explain how one can use Assumption~\ref{A:projectivity} to prove a version of Hopkins' theorem. A proof is included for the reader to 
see where Assumption~\ref{A:projectivity} is used as a replacement for the extension of support data. 

\begin{prop} \label{T:Hopkins} Let $\bK$ be a compactly generated TTC, $X$ be a Zariski space, and 
$\X_{cl}$ be the collection of closed subsets of $X$. Let $V:\bK^{c}\to \X_{cl}$ be a quasi support data defined on $\bK^{c}$ satisfying the additional condition (\ref{SS:supportdata}.\ref{E:supportseven}), with Assumption~\ref{A:projectivity} also holding. Given 
$M\in\bK^{c}$, set $W= V(M)%= \bigcup_{\{ Z\in\XX \mid  Z \subseteq V(M) \}} Z
$.
Then $\bI_{W}=\Tensor(M)$.
\end{prop}

\begin{proof} First note that if $M=0$ then $V(M)=\varnothing=W$. Moreover, $\Tensor(M)$ and $I_\varnothing$ just consist of $0$, by (\ref{SS:supportdata}.\ref{E:supportseven}). Therefore, we can assume that $M\neq 0$.  
Set $\bI=\bI_{W}$ and $ \bI'=\Tensor(M)$. 

\medskip
\noindent ($\supseteq$) Since $\Tensor(M)$ is the smallest thick tensor ideal of $\bK^{c}$ containing $M$, from the definition of $\bI_{W}$ and the properties of a quasi support data it follows that $\bI \supseteq \bI'$.

\medskip
\noindent ($\subseteq$) Let $N\in\bK$. Apply the exact triangle of functors $\Gamma_{\bI'}\to\text{Id} \to L_{\bI'}\to\ $ to $\Gamma_{\bI}(N)$:
\begin{equation} \label{E:ExactTriangle}
\Gamma_{\bI'}\Gamma_{\bI}(N) \to \Gamma_{\bI}(N) \to L_{\bI'} \Gamma_{\bI}(N) \to
\end{equation}
Since $\bI' \subseteq \bI$, the first term belongs to $\Loc(\bI') \subseteq \Loc(\bI)$. The second term also belongs to $\Loc(\bI)$, a triangulated subcategory, and hence so does the third term: $L_{\bI'}\Gamma_{\bI}(N) \in \Loc(\bI)$. 

By \cref{T:localizationtriangles} (applied to $\bI'$), there are no non-zero maps from $\bI'$ to $L_{\bI'}\Gamma_{\bI}(N)$. Thus, for any object $S\in\bK^{c}$, the duality property implies 
that 
\begin{equation} \label{E:ExactNTriangle}
0=\Hom_{\bK}(M\otimes S, L_{\bI'}\Gamma_{\bI}(N)) \cong \Hom_{\bK}(S, M^{*}\otimes L_{\bI'}\Gamma_{\bI}(N)).
\end{equation}
Since $\bK$ is compactly generated it follows that 
$M^{*}\otimes L_{\bI'}\Gamma_{\bI}(N)=0$ in $\bK$. 
 
Now specialize to $N\in\bI$. Then we have $\Gamma_{\bI}(N) \cong N$, and $M^{*}\otimes L_{\bI'}(N)=0$ in $\bK$. Therefore, by Assumption~\ref{A:projectivity},  
$L_{\bI'}\Gamma_{\bI}(N)= L_{\bI'}(N)=0$ in $\bK$. By \cref{E:ExactTriangle} it follows that $\Gamma_{\bI}(N) \cong \Gamma_{\bI'}\Gamma_{\bI}(N)$, so 
$\Gamma_{\bI'}(N) \cong N$, whence $N\in\operatorname{Loc}\left( \bI'\right)$. Applying \cite[Lemma 2.2]{Nee} we see that in fact $N \in \bI'$.  This shows $\bI\subseteq\bI'$ and completes the proof.
\end{proof} 

\subsection{The Tensor Product Property}\label{SS:tensorproductproperty} One can use Hopkins' Theorem and the Realization Property to prove that a quasi support data has the tensor product property and, hence, is a support data. 
The tensor product property is used in proving Theorem~\ref{K:bijectiongeneral} to show that the map $f$ takes points in $X$ to prime ideals. 

\begin{lemma}\label{L:TensorProductonIdeals} Let $\bK$ be a TTC and let $M,N \in \bK^{c}$ be fixed.  For any objects $A \in \Tensor (M)$ and $B \in \Tensor (N)$ one has $A \otimes B \in \Tensor (M \otimes N)$ 
\end{lemma}

\begin{proof} By definition $\Tensor (Q)$ is the smallest tensor ideal of $\bK^{c}$ containing $Q \in \bK^{c}$.  That is, any object of $\Tensor (Q)$ can be obtained from $Q$ by a finite sequence of the following operations:
\begin{enumerate}
\item If $A \in \Tensor (Q)$, then $\Sigma A \in \Tensor (Q)$;
\item If $A \to B \to C \to $ is a triangle in $\bK$ and two of $A,B,C$ are objects in $\Tensor(Q)$, then the third is an object of $\Tensor (Q)$;
\item If $A \in \Tensor (Q)$ and $B \in \bK$, then $A \otimes B \in \Tensor (Q)$;
\item If $A \oplus B$ is an object of $\Tensor(Q)$, then $A,B$ are objects of $\Tensor (Q)$.
\end{enumerate}  Given $A \in \Tensor (Q)$, write $\ell_{Q}(A) \in \Z_{\geq 0}$ for the minimal number of operations required to obtain $A$ from $Q$. 

The result is proved by induction on $\ell_{M}(A) + \ell_{N}(B)$.  If this sum equals zero, then $A \cong M$ and $B \cong N$ and the result holds trivially.  Now assume that $A \in \Tensor (M)$ and $B \in \Tensor (N)$ are given.  Without loss of generality assume $\ell_{M}(A) > 0$.  Then there are $A', A'' \in \Tensor (M)$ such that $\ell_{M}(A'),\, \ell_{M}(A'') < \ell_{M}(A)$ and  $A$ is obtained from $A'$ and $A''$ by one of the above operations.  If $A = \Sigma A'$, then  $A' \otimes B \in \Tensor (M \otimes N)$ and $A \otimes B = (\Sigma A') \otimes B \cong \Sigma (A' \otimes B)$, where the isomorphism holds by the defining properties of a TTC.  Thus $A \otimes B \in \Tensor (M \otimes N)$.  If $A_{1} \to A_{2} \to A_{3} \to$ is a distinguished triangle in $\bK$ with $\{A, A', A'' \} = \{A_{1}, A_{2}, A_{3} \}$, then one can apply the exact functor $- \otimes B$ to obtain the triangle $A_{1} \otimes B \to A_{2}\otimes B \to A_{3}\otimes B \to $.  By induction two of the three terms lie in $\Tensor (M\otimes N)$ and, hence, so does $A \otimes B$.  If $A = A' \otimes C$ for some $C$, then $A \otimes B = (A' \otimes C) \otimes B \cong (A'\otimes B)\otimes C$.  This along with the fact that $A' \otimes B$ lies in $\Tensor (M\otimes N)$ implies $A \otimes B$ also lies in $\Tensor (M \otimes N)$.  Similarly, if $A' = A \oplus C$ for some $C$ in $\bK^{c}$, then $A'\otimes B = (A \oplus C) \otimes B \cong (A \otimes B) \oplus (C \otimes B)$ lies in $\Tensor (M \otimes N)$ and so $A \otimes B$ does as well.  Therefore, the inductive step holds and the result is proven.
\end{proof}

\begin{prop}\label{P:TensorProductProperty}  Let $\bK$ be a compactly generated TTC, $X$ be a Zariski space, and let $V:\bK^{c}\to \X_{cl}$ be a quasi support data defined on $\bK^{c}$.  If $V$ satisfies both the statement of Hopkins' Theorem and  (\ref{SS:supportdata}.\ref{E:supporteight}), then 
\[
V(M \otimes N) = V(M) \cap V(N)
\] for any $M$ and $N$ in $\bK^{c}$. 
\end{prop}
\begin{proof} Since $V$ is a quasi-support data only the inclusion $ V(M) \cap V(N) \subseteq V(M \otimes N) $ needs to be proved.  Since $V(M)$ and $V(N)$ are closed sets, the intersection $V(M) \cap V(N)$ is a closed set.  Using the Realization Property, fix a $T \in \bK^{c}$ with $V(T)=V(T^{*})=V(M) \cap V(N)$.  Since $V(T)  \subseteq V(M)$, Hopkins' Theorem implies $T \in \Tensor (M)$.  Similarly, $T^{*} \in \Tensor (N)$.  By \cref{L:TensorProductonIdeals} it follows that $T \otimes T^{*} \in \Tensor (M \otimes N)$ and so $T \otimes T^{*}\otimes T$ is an object of $\Tensor (M\otimes N)$.  By \cite[Lemma A.2.6]{HPS}, $T$ is a direct summand of $T \otimes T^{*} \otimes T$. That is, $T$ is an object of $\Tensor (M\otimes N)$ and so $V(T) \subseteq V( M \otimes N)$. That is, $V(M) \cap V(N) \subseteq V(M\otimes N)$ as required.
\end{proof}

\subsection{Thick Tensor Ideals in a TTC and Spc} \label{SS:classifying} One can now use the version of Hopkins' Theorem given in the previous section (\cref{T:Hopkins}) 
and apply the same arguments as in \cite[Theorem 3.4.1]{BKN} to obtain the following classification of thick tensor ideals of compact objects in our setting.  

\begin{theorem} \label{I:bijectiongeneral} Let $\bK$ be a compactly generated TTC.  Let $X$ be a Zariski space, $\X_{cl}$ be the set of all closed subsets of $X$, and let
$V:\bK^{c} \rightarrow \XX_{cl}$ be a quasi support data defined on $\bK^{c}$ satisfying the additional conditions (\ref{SS:supportdata}.\ref{E:supportseven}) and (\ref{SS:supportdata}.\ref{E:supporteight}).  Assume Assumption~\ref{A:projectivity} holds. 

Given the above setup there is a pair of mutually inverse maps
$$
\{\text{thick tensor ideals of $\bK^{c}$}\}  \begin{array}{c} \gf{\Gamma}{\longrightarrow} \\ \gf{\longleftarrow}{\Theta} \end{array} \XX_{sp},
$$
given by 
\begin{align*}
\Gamma(\bI)=\bigcup_{M\in \bI} V(M),\ \ \ \ 
 \Theta(W) = \bI_{W}.
\end{align*}
\end{theorem}

By \cref{P:TensorProductProperty} a quasi support data which satisfies the assumptions of the previous theorem is a support data.  One can argue as in the proof of \cite[Theorem 3.5.1]{BKN} to obtain the following description of the Balmer spectrum of $\bK^{c}$.

\begin{theorem} \label{K:bijectiongeneral} Let $\bK$ be a compactly generated TTC, let $X$ be a Zariski space, and let $\X_{cl}$ be the set of all closed subsets of $X$. Assume that  
$V:\bK^{c} \rightarrow \XX_{cl}$ is a quasi support data defined on $\bK^{c}$ satisfying the additional conditions (\ref{SS:supportdata}.\ref{E:supportseven}) and (\ref{SS:supportdata}.\ref{E:supporteight}). Assume Assumption~\ref{A:projectivity} holds.  Then there is a homeomorphism
$$
f: X  \to \operatorname{Spc}(\bK^{c}).
$$
\end{theorem}

\section{Quantum Groups and their Module Categories} 

\subsection{Notation}  We will follow the conventions as described in 
\cite[Section 2]{BNPP}.

\begin{enumerate}

\item $G$: a simple, simply connected algebraic group over $\C $. The assumption of $G$ being simple is for convenience and the results easily generalize to $G$ reductive.

\item $\fg=\operatorname{Lie }G$.

\item $\UU(\fg)$: the universal enveloping algebra of $\fg$. 

\item $\NN$: the nullcone (i.e., the set of nilpotent elements of $\fg$).

\item $\C [\NN]$: the coordinate algebra of $\NN$.

\item $T$: a maximal split torus in $G$. 

\item $\Phi$: the corresponding (irreducible) root system associated to $(G,T)$. When a root system has roots of only one length, all roots shall be considered both short and long.

\item $\Phi^{\pm}$: the positive (respectively, negative) roots.  

\item $B$: a Borel subgroup containing $T$ corresponding to the negative roots. 

\item $\Pi = \{\alpha_1,\alpha_2,\dots,\alpha_{|\Pi|}\}$: an ordering of the simple roots.

\item $\alpha_{0}$: the highest short root.

\item  $\tilde{\alpha}$: the highest root.

\item $\mathbb E$: the Euclidean space $\R\Phi$, with inner product $\langle-,-\rangle$.

\item $\alpha^{\vee} = 2\alpha/\langle\alpha,\alpha\rangle$ for $0\ne\alpha\in{\mathbb E}$.

\item  $X(T)=\mathbb Z \omega_1\oplus\cdots\oplus{\mathbb Z}\omega_{|\Pi|}$: the weight lattice, where the 
fundamental dominant weights $\omega_i\in{\mathbb E}$ are defined by $\langle\omega_i,\alpha_j^\vee\rangle=\delta_{ij}$, $1\leq i,j\leq |\Pi|$. 

\item $X_{+}={\mathbb N}\omega_1+\cdots+{\mathbb N}\omega_{|\Pi|}$: the dominant weights.

\item $\ell$: a fixed odd positive integer. If $\Phi$ has type $G_2$, then we assume that $3 \nmid \ell$.

\item $X_{1}=X(T)/\ell X(T)$.

\item  $W$: the Weyl group of $\Phi$. 

\item $W_{\ell}=W\ltimes {\ell }{\mathbb Z}\Phi$: the affine Weyl group of $\Phi$. 

\item $\rho$: the Weyl weight defined by $\rho=\frac{1}{2}\sum_{\alpha\in\Phi^+}\alpha$.

\item  $h$: the Coxeter number of $\Phi$, given by $h=\langle\rho,\alpha_0^{\vee} \rangle+1$. 

\setcounter{mytempcount}{\value{enumi}}

\end{enumerate} 
\medskip 

\noindent
For quantum groups we will use the following notation.
\medskip 

\begin{enumerate} 

\setcounter{enumi}{\value{mytempcount}}

\item $\zeta$: a primitive $\ell$th root of unity in $\C$. 

\item $\AA = \Z[q,q^{-1}]$, where $q$ is an indeterminate.

\item  $\BU_{q}(\fg )$:  the quantized enveloping algebra over $\mathbb{Q}(q)$.

%\item $\UU_{\zeta}(\fg)$: the quantized enveloping algebra (specialized at $q=\zeta$ from the 
%de Concini--Kac quantum group). 

\item  $\BU^{\AA}_{q}(\fg)$: Lusztig's $\AA$-form in $\BU_{q}(\fg )$.

\item $U_{\zeta}(\fg)$: the (big) quantum group obtained from the specialization at $\zeta$ of Lusztig's $\AA$-form as described below. 

\item $u_{\zeta}(\fg)$: the small quantum group. 
\item $U^{0}$: the subalgebra of $U_{\zeta}(\fg )$ generated by the binomials in the $K_{\alpha}$'s, $\alpha\in\Pi$.

\item $T(\lambda)$: the (indecomposable) tilting $U_{\zeta}(\fg)$-module of highest weight $\lambda\in X_{+}$. 

\end{enumerate}

In more detail, let $\BU_{q}(\fg )$ be the quantized enveloping algebra defined by generators and relations over $\mathbb{Q}(q)$ (where $q$ is an indeterminate) as in, for example, \cite[Section 4.3]{Jan}.  Let $\BU_{q}^{\AA}(\fg )$ denote the $\AA$-form of $\BU_{q}(\fg )$ defined by Lusztig.  One can then construct the restricted specialization ${\mathbb C} \otimes_{\AA} \BU_{q}^{\AA}(\fg )$ by specializing $q$ to $\zeta\in \mathbb{C}$.  Within this algebra $1 \otimes K_{\alpha}^{\ell}$ 
is central for all $\alpha \in \Pi$.  By definition a module for the restricted specialization is of Type 1 if it is annihilated by $1 \otimes K_{\alpha}^{\ell}-1\otimes 1$ for all $\alpha \in \Pi$.  The discussion in \cite[Section 1.6]{APW} implies that for the modules and questions under consideration here we may assume all modules for the restricted specialization are of Type 1.  That is, from this point on we study modules for $U_\zeta(\fg)$ where
\begin{equation*}\label{firstquantumgroup}
U_\zeta(\fg):=\C \otimes_{\AA}\BU^{\AA}_q(\fg)/(  1\otimes 
K_\alpha^\ell-1\otimes 1 \mid \alpha\in\Pi ).
\end{equation*}
 Here $(  \cdots )$ denotes the two-sided ideal generated by the elements within the parentheses. Following Lusztig, we view $\BU_{q}(\fg )$ as a Hopf algebra using the maps given in \cite{APW,BNPP,Jan}.  These maps induce a Hopf algebra structure on the restricted specialization and on $U_{\zeta}(\fg)$.

The elements $E_\alpha,K_\alpha,F_\alpha$, $\alpha\in\Pi$, in $U_{\zeta}(\fg)$ (i.e., the images of $1 \otimes E_{\alpha}$, etc.\ in $U_{\zeta}(\fg )$) generate a
finite-dimensional Hopf subalgebra, denoted by $u_\zeta(\fg)$, of $U_\zeta(\fg)$.
The Hopf algebra $u_{\zeta}(\fg)$ will be referred to as the {\em small quantum group}. The finite-dimensional algebra $u_{\zeta}(\fg)$ is 
a normal  Hopf subalgebra of $U_{\zeta}(\fg)$ such that 
$U_{\zeta}(\fg)/\!\!/u_{\zeta}(\fg)\cong \UU(\fg)$, 
%where $\UU(\fg)$ is the ordinary universal enveloping algebra $\fg$  
\cite{lusztig:finite-dim, lusztig:root}.  We write $u_{\zeta}(\ft )$, $u_{\zeta}(\fu )$, and $u_{\zeta}(\fb )$ for the subalgebras of $u_{\zeta}(\fg )$ 
generated by $\left\{K_{\alpha}^{\pm 1} \mid \alpha \in \Pi \right\}$,  $\left\{E_{\alpha} \mid \alpha \in \Pi \right\}$, and  $\left\{K_{\alpha}^{\pm 1}, E_{\alpha} \mid \alpha \in \Pi \right\}$, respectively. Sometimes we write $K_{i}$ for $K_{\alpha_{i}}$.

\subsection{Module Categories}\label{SS:modulecategories}

 Let $\Mod \left(U_{\zeta}(\fg ) \right)$ be the full subcategory of all $U_{\zeta}(\fg )$-modules which admit a weight space decomposition for $U^{0}$ with weights lying in $X(T)$ and which are locally finite as $U_{\zeta}(\fg)$-modules. This is precisely the category of integrable Type 1 modules of \cite{BNPP} and \cite{APW}.  Furthermore, let $\mod \left(U_{\zeta}(\fg ) \right)$ denote the category of all finite-dimensional $U_{\zeta}(\fg )$-modules.   By \cite[Theorem 9.12(i)]{APW} every finite-dimensional $U_{\zeta}(\fg)$-module admits a weight space decomposition for $U^{0}$ with weights lying in $X(T)$;  hence,  $\mod \left(U_{\zeta}(\fg ) \right)$ is a full subcategory of  $\Mod \left(U_{\zeta}(\fg ) \right)$.   

Let $A$ denote $u_{\zeta}(\fh )$ or $u_{\zeta}(\fb )$.  Then $\Mod (A)$ denotes the full subcategory of all $A$-modules which admit a weight space decomposition with respect to $u_{\zeta}(\fh )$ with weights lying in $X_{1}$.  Let $\mod (A)$ denote the full subcategory of $\Mod (A)$ consisting of the finite-dimensional modules. In the case of $u_{\zeta}(\fu )$ we let $\Mod (u_{\zeta}(\fu ))$ denote the category of all $u_{\zeta}(\fu)$-modules and $\mod (u_{\zeta}(\fu ))$ the full subcategory of all finite-dimensional $u_{\zeta}(\fu )$-modules.  Note that if $A$ and $B$ are any of $U_{\zeta}(\fg )$, $u_{\zeta}(\fg)$, $u_{\zeta}(\fb )$, $u_{\zeta}(\fh)$, or $u_{\zeta}(\fu)$ and $B$ is a subalgebra of $A$, then restriction defines functors from $\Mod (A)$ to $\Mod (B)$ and from $\mod (A)$ to $\mod (B)$.  Our general convention is that by $A$-module we mean an object of $\Mod (A)$ and by finite-dimensional $A$-module we mean an object of $\mod (A)$.

The Steinberg module is a self-dual simple module which is projective and injective in both $\Mod (U_{\zeta}(\fg ))$ and $\mod (U_{\zeta}(\fg ))$.  From this it follows that these categories have enough projectives and injectives and they coincide  (see \cite[Section 9]{APW} and \cite{APW2} for details).  Since the projectives and injectives coincide one can form the stable module categories for $\Mod (U_{\zeta}(\fg ))$ and $\mod (U_{\zeta}(\fg ))$. Recall that these have the same objects, but one quotients out the morphisms which factor through a projective module. One thereby obtains the triangulated categories $\bK:=\Stmod (U_{\zeta}(\fg ))$ and $\bK^{c}:=\stmod (U_{\zeta}(\fg ))$, respectively. Since the objects of $\mod (U_{\zeta}(\fg ))$ are finite dimensional it immediately follows that the subcategory $\bK^{c}$ of $\bK$ consists of compact objects.  Using the fact that the objects of $\Mod (U_{\zeta}(\fg ))$ are locally finite one can argue just as in \cite[Proposition 4.2.1]{BKN} to verify that $\bK$ is compactly generated and $\bK^{c}$ is precisely the triangulated subcategory of compact objects. 

%Note that $\mathcal{C}_{\Gamma}$ in [APW] is the category of all integrable modules.

\section{Projectivity Results for Quantum Groups} 

\subsection{Projectivity Results I} In the following two sections we will generalize projectivity tests proved for finite-dimensional $U_{\zeta}(\fg)$-modules 
to arbitrary modules in $\Mod(U_{\zeta}(\fg))$.  In particular, all modules under discussion will be assumed to be Type $1$ and integrable. Support varieties were used in the proofs for the analogous results in the finite-dimensional cases (cf.\ \cite[Theorem (1.2)]{FP}, \cite[Theorem 5.2.1]{Dr}). Our more general results will enable us to verify Assumption~\ref{A:projectivity} for $\bK=\Stmod (U_{\zeta}(\fg))$. 

Given a set $J$ of simple roots we adopt the notation of \cite[Section 2.5]{BNPP} and write $\fp_{J}$ for the corresponding parabolic subalgebra of $\fg$, $u_{\zeta}(\fp_{J})$ for the corresponding small quantum group viewed as a subalgebra of $u_{\zeta}(\fg)$, etc.  

\begin{theorem} \label{T:proj1} Let $Q$ be in $\Mod(U_{\zeta}(\fg))$ and $J\subseteq \Pi$. Then $Q$ is projective as a $u_{\zeta}(\fg)$-module 
if and only if $Q$ is projective as a $u_{\zeta}(\fp_{J})$-module. 
\end{theorem} 

\begin{proof} Let $Q$ be projective as a $u_{\zeta}(\fg)$-module. Then 
$Q$ is projective as a $u_{\zeta}(\fp_{J})$-module since $u_{\zeta}(\fg)$ is 
free as a $u_{\zeta}(\fp_{J})$-module. 

On the other hand, suppose that $Q$ is projective as a $u_{\zeta}(\fp_{J})$-module, and let 
$N$ be a simple module in $\mod(u_{\zeta}(\fg))$. Since $N$ can be regarded as a module in $\mod(U_{\zeta}(\fg))$, 
there exists a spectral sequence (cf.\ \cite[Theorem 5.1.1]{BNPP}): 
\begin{equation} 
E_{2}^{i,j}=R^{i}\Ind_{P_{J}}^{G}\Ext^{j}_{u_{\zeta}(\fp_{J})}(Q,N)^{(-1)} \Rightarrow\Ext^{i+j}_{u_{\zeta}(\fg)}(Q,N)^{(-1)}. 
\end{equation} 
Here $(-1)$ indicate untwisting by the Frobenius morphism so that the cohomology becomes a rational $G$-module. Since $Q$ is projective as a $u_{\zeta}(\fp_{J})$-module this spectral sequence collapses and yields 
\begin{equation} 
R^{i}\Ind_{P_{J}}^{G}\Hom_{u_{\zeta}(\fp_{J})}(Q,N)^{(-1)} \cong \Ext^{i}_{u_{\zeta}(\fg)}(Q,N)^{(-1)}
\end{equation} 
for $i\geq 0$. By the Grothendieck vanishing theorem, 
$$R^{i}\Ind_{P_{J}}^{G}\Hom_{u_{\zeta}(\fp_{J})}(Q,N)^{(-1)}=0$$ 
for $i> \dim G/P_{J}=D:=|\Phi^{+}|-|\Phi_{J}^{+}|$. Therefore, $\Ext^{i}_{u_{\zeta}(\fg)}(Q,N)=0$ 
for $i>D$. 

Let $\dots \rightarrow P_{n}\rightarrow P_{n-1} \rightarrow \dots \rightarrow P_{0} \rightarrow Q \rightarrow 0$ be a minimal projective resolution 
of $Q$ in $\Mod(U_{\zeta}(\fg))$. At each stage we have a short exact sequence
$$0\rightarrow \Omega^{n+1}(Q)\rightarrow P_{n} \rightarrow \Omega^{n}(Q) \rightarrow 0.$$ 
By dimension shifting we see that for $m\geq 1$, 
$$0=\Ext^{m+D}_{u_{\zeta}(\fg)}(Q,N)\cong\Ext^{m}_{u_{\zeta}(\fg)}(\Omega^{D+1}(Q),N).$$ 
By interpreting this statement in the stable module category (with the fact that the compact objects are the finite-dimensional $u_{\zeta}(\fg)$-modules), one can apply \cite[Theorem 2]{BIK} to 
conclude that $\Omega^{D+1}(Q)$ is projective as $u_{\zeta}(\fg)$-module. Since projective modules are 
injective in $\Mod(u_{\zeta}(\fg))$ the sequence 
$$0\rightarrow \Omega^{D+1}(Q)\rightarrow P_{n} \rightarrow \Omega^{D}(Q) \rightarrow 0$$ 
splits. Since $P_{n}$ is projective as a $u_{\zeta}(\fg )$-module, this shows that $\Omega^{D}(Q)$ is projective as $u_{\zeta}(\fg)$-module. By iterating this process 
we can conclude that $Q$ is projective as $u_{\zeta}(\fg)$-module. 
\end{proof} 

\subsection{}  \label{SS:grb} Henceforth we set $A = u_{\zeta}(\fb)$. For this subsection we let $N = |\Phi^+|$. Order the positive roots, $\Phi^{+}=\left\{\gamma_{1}, \dotsc , \gamma_{N} \right\}$ as in \cite[Section 2.4]{BNPP}. 
Then the algebra $u_{\zeta}(\fu)$ has a monomial basis 
$$
\{F_{\ga_{1}}^{a_{1}}F_{\ga_{2}}^{a_{2}}\cdots
F_{\ga_N}^{a_N} \mid 0\leq a_i \leq \ell-1,\ i=1,2,\dots, N\}.
$$
For a shorthand notation, let 
$$
F_{\overline{a}} :=
F_{\ga_{1}}^{a_{1}}F_{\ga_{2}}^{a_{2}}\cdots F_{\ga_N}^{a_N}.
$$
We place a total (lexicographical) ordering $\prec$ on this monomial 
basis as follows. Set $\overline{a} \prec \overline{b}$ if and only if
there exists $1 \leq i \leq N$ such that $a_i < b_i$ and $a_j =
b_j$ for all $j > i$. With this ordering, one can define a filtration on $A$. Given $\overline{a}$, let 
$A_{\overline{a}}$ be the free $u_{\zeta}(\fh )$-module spanned by the $F_{\overline{b}}$ where  $\overline{b} \preceq \overline{a}$.  
%That is, the subspace of $A$
%generated by
%$$
%\{K_{i}: i=1,2,\dots,n\} \cup \{ F_{\overline{b}} :  \overline{b} \preceq \overline{a}\}.
%$$
By \cite[Lemma 2.4.1]{BNPP} we have $A_{\overline{a}\vphantom{+\overline{b}}}\cdot A_{\overline{b}}
\subseteq A_{\overline{a}+\overline{b}}$ and so the filtration is multiplicative. 

The associated graded algebra $\gr  A$ is generated by $\{K_{i} \mid i=1,2,\dots, |\Pi|\}$ and $\{X_{\alpha} \mid \alpha\in
\Phi^{+}\}$ subject to the relations:
\begin{gather*} \label{E:relations}
K_{i}K_{j}=K_{j}K_{i}, \qquad K_{i}X_{\alpha}=\zeta^{\langle \alpha_{i},\alpha\rangle} X_{\alpha}K_{i}, \\
%\end{equation*} 
%\begin{equation*}
X_{\alpha} X_{\beta}=\zeta^{\langle \alpha,\beta
\rangle}X_{\beta}X_{\alpha}\ \ \text{if $\alpha\prec \beta$}.
\end{gather*}
Here we write $\alpha\prec \beta$ to denote that $\alpha = \gamma_{i}$ and $\beta=\gamma_{j}$ with $i<j$.
The generators are also subject to the additional conditions:
\begin{gather*}\label{E:nilpotent}
X_{\alpha}^{\ell}=0\ \ \text{for $\alpha\in
\Phi^{+}$}, \\
%\end{equation*}
%\begin{equation*}\label{nilpotent2}
K_{i}^{\ell}=1\ \ \text{for $i=1,2,\dots,|\Pi|$}.
\end{gather*}

The algebra $\gr  A = \gr u_{\zeta}(\fb)$ is a quantized version of a truncated symmetric algebra which we will denote by $\bar A$.  The same total ordering defines a filtration on the subalgebra $u_{\zeta}(\fu ) \subseteq u_{\zeta}(\fb )$ which, in turn, defines a subalgebra $\gr u_{\zeta}(\fu ) \subseteq \bar A$.  In terms of the above presentation for $\bar A$ this is precisely the subalgebra generated by $\left\{X_{\alpha} \mid \alpha \in \Phi^{+} \right\}$.  See \cite{GK}, for example, for details.  In particular, this subalgebra is a quantum complete intersection in the sense of \cite{BEH}.  We write $A' = u_{\zeta}(\fu)$ and $\bar{A}'=\gr u_{\zeta}(\fu )$ for these algebras.

\subsection{Graded Module Categories}\label{SS:gradedmodulecategories} We observe that the definition of the graded module $\gr M$ (given in Section~\ref{S:gradings}) depends on the choice of generators for $M$. For our purposes we need to choose the generators carefully to insure that projective $A$-modules go to projective $\bar{A}$-modules. This can be done by choosing a basis of elements in the module $M$ modulo its radical, and checking that projective indecomposable $A$-modules go to projective indecomposable $\gr A$-modules under this choice of generators. In particular, by always choosing generators so that projectives go to projectives passing to the associated graded provides a well defined operator on objects between stable module categories. 

Let $\Mod (\bar A )$ be the full subcategory of all $\bar A$-modules which have a weight space decomposition with respect to the commutative subalgebra $\gr u_{\zeta}(\mathfrak h)$ generated by the elements $K_{1}, \dotsc , K_{|\Pi|}$ where all weights lie in $X_{1}$. Moreover, let $\mod (\bar A)$ denote the full subcategory of $\Mod (\bar A )$ consisting of finite-dimensional modules.   Let $\Mod (\gr u_{\zeta}(\fu) )$ denote the category of all $\gr u_{\zeta}(\fu )$-modules and $\mod (\gr u_{\zeta}(\fu ))$ denote the full subcategory of all finite-dimensional $\gr u_{\zeta}(\fu )$-modules.  Our convention is that all modules are assumed to lie in these categories.  Note that if $D$ is $u_{\zeta}(\fb )$ or $u_{\zeta}(\fu)$, then for any module $M$ in $\Mod (D)$ (resp.\ $\mod (D)$) the associated graded module $\gr M$ constructed in \cref{S:graded}  lies in $\Mod (\gr D)$ (resp.\ $\mod (\gr D)$).  Finally, note that in both cases $\gr D$ is self-injective.  Consequently the projective and injective modules coincide and, hence, their stable module categories are triangulated categories.

\subsection{Projectivity Results II} In this section we first analyze the relationships between projectivity for modules of $u_{\zeta}(\fb)$ and $u_{\zeta}(\fu)$ and their  
associated graded algebras.   Given $\lambda \in X_{1}$, let $\lambda$ also denote the one-dimensional simple $u_{\zeta}(\ft )$-module (resp.\  $u_{\zeta}(\fb )$-module) of that weight.  The collection of all such is a complete set of simple $u_{\zeta}(\ft )$-modules (resp.\  $u_{\zeta}(\fb )$-modules).

\begin{prop} \label{P:btou} Let $Q$ be a $u_{\zeta}(\fb)$-module. 
\begin{itemize} 
\item[(a)]  $Q$ is projective as a $u_{\zeta}(\fb)$-module if and only if 
$Q$ is projective as a $u_{\zeta}(\fu)$-module. 
\item[(b)]  $\gr  Q$ is projective as a $\gr  u_{\zeta}(\fb)$-module if and only if 
$\gr  Q$ is projective as a $\gr  u_{\zeta}(\fu)$-module. 
\item[(c)] $Q$ is projective as a $u_{\zeta}(\fu)$-module if and only if $\gr  Q$ is projective as a $\gr  u_{\zeta}(\fu)$-module. 
\item[(d)] $Q$ is projective as a $u_{\zeta}(\fb)$-module if and only if $\gr  Q$ is projective as a $\gr  u_{\zeta}(\fb)$-module. 
\end{itemize} 
\end{prop} 

\begin{proof} (a) The ``only if'' direction of the statement follows because $u_{\zeta}(\fb)$ is a free $u_{\zeta}(\fu)$-module. 
Assume that $Q$ is projective over $u_{\zeta}(\fu)$. The Lyndon-Hochschild-Serre spectral sequence for 
$u_{\zeta}(\fu)\lhd u_{\zeta}(\fb)$ collapses because the quotient is isomorphic to $u_{\zeta}(\ft)$ and by assumption the modules under consideration are semisimple as $u_{\zeta}(\ft )$-modules. This yields the 
isomorphism 
$$\Ext^{i}_{u_{\zeta}(\fb)}(Q,N)\cong\Hom_{u_{\zeta}(\ft)}({\mathbb C},\Ext^{i}_{u_{\zeta}(\fu)}(Q,N))$$ 
for $N$ a $u_{\zeta}(\fb)$-module and $i\geq 0$. Since $Q$ is projective over $u_{\zeta}(\fu)$ we can conclude from this isomorphism that 
$\Ext^{i}_{u_{\zeta}(\fb)}(Q,N)=0$ for $i>0$. Therefore, $Q$ is projective as a $u_{\zeta}(\fb)$-module. 

(b) This follows from the same line of reasoning as in (a). 

(c)  If $Q$ is projective as a $u_{\zeta}(\fu)$-module then $Q$ is a direct sum of copies of $u_{\zeta}(\fu)$, and thus 
$\gr Q$ is projective as a $\gr  u_{\zeta}(\fu)$-module. 

For the converse, we have an increasing multiplicative filtration on $u_{\zeta}(\fu)$ which leads to a spectral sequence from Proposition~\ref{P:Mayspectral}: 
$$E_{1}^{i,j}=\HH^{i+j}(\gr  u_{\zeta}(\fu), {}\gr Q)_{(i)}\Rightarrow \HH^{i+j}(u_{\zeta}(\fu),Q).$$ 
Since $\gr  Q$ is projective as a $\gr  u_{\zeta}(\fu)$-module we have $E_{1}^{i,j}=0$ for $i+j\geq 1$. This shows that 
$\HH^{i+j}(u_{\zeta}(\fu),Q)=0$ for $i+j\geq 1$. Consequently, $Q$ is projective as a $u_{\zeta}(\fu)$-module. 

(d) This is a consequence of parts (a), (b), and (c). 
\end{proof} 

\section{Localization and Supports} 

\emph{For the remainder of the paper we will assume that $\zeta$ is a primitive $\ell$th root of unity where 
$\ell>h$}. Throughout this section set $A=u_{\zeta}({\mathfrak b})$, $\bar{A}=\gr u_{\zeta}({\mathfrak b})$, $A^{\prime}= u_{\zeta}(\fu)$, and $\bar{A}^{\prime}=\gr u_{\zeta}(\fu)$. 

\subsection{}\label{SS:BIKforAbar} 

 Let $R=\HH^{\bullet}(A,\C)=\HH^{\bullet}(u_{\zeta}(\fb),\C )$.  Since $\ell>h$ we have $R\cong S^{\bullet}(\fu^{*})$ as an algebra by \cite[Theorem 2.5]{GK}.  Now consider the spectral sequence in Proposition~\ref{P:Mayspectral} with $M=\C$.  As in the proof of \cite[Theorem 2.5]{GK}, the spectral sequence collapses and so we have the isomorphisms: 
$$S^{\bullet}(\fu^{*}) \cong R=\HH^{2\bullet}(u_{\zeta}(\fb),\C )\cong \HH^{2\bullet}(\gr u_{\zeta}(\fb),\C ) .$$ 
These are isomorphisms of graded rings where the generators of $ S^{\bullet}(\fu^{*})$ are in degree two. See \cref{SS:relatingvarieties} for more details.

Let $R^{\prime}=\HH^{\bullet}(u_{\zeta}(\fu),\C )$. Then for any 
$u_{\zeta}(\fu)$-module $M$, $R^{\prime}$ acts on $\text{Ext}^{\bullet}_{u_{\zeta}(\fu)}(\C,M)$ via the Yoneda product. Since $R$ is a subring of $R^{\prime}$, it 
follows that $R$ acts on $\text{Ext}^{\bullet}_{u_{\zeta}(\fu)}(\C,M)$. For any $u_{\zeta}({\mathfrak b})$-module $M$, one can apply the Lyndon-Hochschild-Serre spectral 
sequence to prove that 
$$\Ext^{\bullet}_{u_{\zeta}(\fb)}(\oplus_{\lambda\in X_{1}}\lambda,M)\cong\Ext ^{\bullet}_{u_{\zeta}(\fu)}(\C,M).$$
In this way, $R$ naturally acts on $\Ext^{\bullet}_{u_{\zeta}(\fb)}(\oplus_{\lambda\in X_{1}}\lambda,M)$. A similar statement on actions can be made when $A$ is replaced by $\bar{A}$ and 
$A^{\prime}$ is replaced by $\bar{A}^{\prime}$. 

Set $X=\Proj(R)=\Proj(\Spec(R))$. Given an ideal $I$ of $R$ we write $\V( I)=\{P'\in X\mid\ P'\supseteq I\}$ for the closed subset of $X$ determined by $I$.  
The \emph{specialization closure} of a subset $U \subseteq X$ is the subset $\cl (U)=\cup_{P \in U} \V (P)$.  A subset $U \subseteq X$ is called \emph{specialization closed} if $\cl (U)=U$; that is, if $U$ is the union of closed sets. 

For $M\in\Stmod (A)$, let 
\begin{align*}
\ZZ_{A}(M)&=\{P \in X\mid \Ext^{\bullet}_{u_{\zeta}(\fb)}(\oplus_{\lambda\in X_{1}}\lambda,M)_{P}\neq 0\} \\
                               &=\{P \in X \mid\Ext^{\bullet}_{u_{\zeta}(\fu)}(\C ,M)_{P}\neq 0\}. 
\end{align*}

For $M\in\Stmod (\bar A)$ one can similarly define $\ZZ_{\bar A}(M)$. 
Now applying Proposition~\ref{P:Mayspectral}, and the fact that $R$ acts on the spectral sequence, one sees that for 
$M\in\Stmod(u_{\zeta}(\fb))$, 
\begin{equation} \label{E:AvarietyinbarAvariety}
\ZZ_{A}(M)\subseteq \ZZ_{\bar A}(\gr M). 
\end{equation}  
One can verify directly that $\ZZ_{A}(-)$ and $\ZZ_{\bar A}(-)$ satisfy (\ref{SS:supportdata}.\ref{E:supportone})--(\ref{SS:supportdata}.\ref{E:supportfour}); we 
use these properties without comment in what follows.

\subsection{\texorpdfstring{A Tensor Product for $\bar{A}$-Modules}{A Tensor Product for (gr A)-Modules}} \label{SS:TensorAbarmodules}  
Let $\Delta: u_{\zeta}(\fb ) \to u_{\zeta}(\fb )\otimes u_{\zeta}(\fb )$ be the coproduct for $u_{\zeta}(\fb )$ and let $\varepsilon: u_{\zeta}(\fb ) \to \C$ be the counit. From the defining properties of a Hopf algebra the composition
\begin{equation}\label{E:coalgebramaps}
u_{\zeta}(\fb ) \xrightarrow{\Delta} u_{\zeta}(\fb) \otimes u_{\zeta}(\fb ) \xrightarrow{1 \otimes \varepsilon} u_{\zeta}(\fb ) \otimes \C \xrightarrow{m} u_{\zeta}(\fb )
\end{equation} 
is the identity, where here $m$ denotes the multiplication map. Consequently, $\Delta$ is injective and we can identify $u_{\zeta}(\fb )$ as a subalgebra of $ u_{\zeta}(\fb )\otimes u_{\zeta}(\fb )$ via $\Delta$.  From \cref{SS:grb} there is a filtration of $u_{\zeta}(\fb)$ which induces a filtration on $u_{\zeta}(\fb ) \otimes u_{\zeta}(\fb )$ and on the image of $\Delta$ (see \cref{SS:associatedgraded}). This in turn induces (via $\Delta$) a different filtration on $u_{\zeta}(\fb)$, which (by construction) makes $\Delta$ a map of filtered algebras.  %Passing to the associated graded gives an algebra map $\gr \Delta$. 
We write $\Delta$ for the subalgebra $\gr \Delta (u_{\zeta}(\fb )) \subseteq \gr(u_{\zeta}(\fb )\otimes u_{\zeta}(\fb ))$.

Given $\bar{A}$-modules $Q_{1}$ and $Q_{2}$ we write $Q_{1}\otimes Q_{2}$ for the $\bar{A}\otimes\bar{A}$-module given by taking their outer tensor product. Then $Q_{1}\otimes Q_{2}$ can be viewed as a $\Delta$-module by restriction to this subalgebra. For example, if $\lambda$ and $\mu$ are simple $\bar{A}$-modules, then $\lambda \otimes \mu$ is a one-dimensional simple $\Delta$-module.  In particular, in this way $\mathbb{C} = \mathbb{C} \otimes \mathbb{C}$ is a module for $\Delta$. 
Furthermore, several times in the sequel the following observation is needed.  If $M$, $N$ are $A$-modules and $\gr \left(M \otimes N \right)$ is projective as an $\bar{A}$-module, then $M \otimes N$ is a projective $A$-module by \cref{P:btou}.  In turn, $\gr M \otimes \gr N$ is projective as a $\Delta$-module by the discussion at the beginning of \cref{SS:gradedmodulecategories}.  
 
Finally, note that the antipode on $A$ preserves the filtration given in \cref{SS:grb} and induces an anti-automorphism on $\bar{A}$.  In particular, given a finite-dimensional $\bar{A}$-module $M$ we can consider the corresponding dual module $M^{*}$.

\subsection{}\label{SS:BIKsupports} Given a finite-dimensional $\bar{A}$-module $Q_{1}$ and an arbitrary $\bar{A}$-module $Q_{2}$, $\Ext_{\bar{A}}^{\bullet}(Q_{1},Q_{2})$ can be made into an $R$-module in the following way. From 
Theorem~\ref{thm:injectivemap}, we have the isomorphism, 
$$\text{Ext}^{\bullet}_{\bar{A}}(Q_{1},Q_{2})\cong\Ext ^{\bullet}_{\Delta}({\mathbb C},Q_{1}^{*}\otimes Q_{2}).$$ 
One can then use the natural action of $R\cong \HH^{2\bullet}(\Delta,\C)$ on $\text{Ext}^{\bullet}_{\Delta}({\mathbb C},Q_{1}^{*}\otimes Q_{2})$ to put an action of $R$ 
on $\text{Ext}^{\bullet}_{\bar{A}}(Q_{1},Q_{2})$. For $Q_{1}=\oplus_{\lambda\in X_{1}}\lambda$, this coincides with the aforementioned construction of the action of $R$ on 
$\Ext_{\bar{A}}^{\bullet}(\oplus_{\lambda\in X_{1}}\lambda,Q_{2})$. In the case when $Q_{1}=Q_{2}$, we have the map of rings: 
\begin{equation}\label{E:ringmap}
R\rightarrow\Ext ^{\bullet}_{\Delta}({\mathbb C},Q_{1}^{*}\otimes Q_{1})\cong\Ext ^{\bullet}_{\bar{A}}(Q_{1},Q_{1}).
\end{equation} 

The results in \cite[Section 5]{BIK} use the assumption that one has a ring map (\ref{E:ringmap}) for an arbitrary module $Q_{1}$. However, for our setting we only have 
an action of $R$ on $\text{Ext}^{\bullet}_{\bar{A}}(Q_{1},Q_{2})$ when $Q_{1}$ is finite-dimensional  (i.e., compact). Using this action, the main results (i.e., \cite[Theorems 5.2, 5.13]{BIK}) 
will hold for $\bar{A}$-modules. The verification of \cite[Theorem 5.13]{BIK} for $\bar{A}$ heavily uses the aforementioned action, and the fact that 
$Q_{1}$ is compact. Furthermore, the map (\ref{E:ringmap}) is used to construct Koszul objects (cf. \cite[Definition 5.10]{BIK}) for finite-dimensional $\bar{A}$-modules. 

For each $M$ in $\Stmod  \left(A \right)$ (resp.\ $\Stmod \left( \bar{A}\right)$) and $P\in X=\Proj (R)$, let $\nabla_{P}(M)$ be the object\footnote{Denoted by $\Gamma_{P}M$ in \cite{BIK}.} in $\Stmod (A)$ (resp.\ $\Stmod (\bar{A}$)) constructed via (co)localization functors in \cite[Section 5]{BIK}.  The following theorem gives an alternate description of $\nabla_{P}(M)$.

\begin{theorem}\label{T:NablaAsTensor} Let $M \in \Stmod (A)$ and $P \in X$. Then, 
\[
\nabla_{P}(M) \cong M \otimes \nabla_{P}(\C )
\] as $A$-modules.

% Similarly, let  $M \in \Stmod (\bar{A})$ and $P \in X$. Then, 
%\[
%\nabla_{P}(M) \cong M \boxtimes \nabla_{P}(\C )
%\] as $\bar{A}$-modules.

\end{theorem}

\begin{proof} The arguments which prove the analogous \cite[Corollary 8.3]{BIK} apply here as well with a few minor modifications to the proofs of \cite[Proposition 8.1]{BIK} and \cite[Theorem 8.2]{BIK}.  In particular, as the tensor product may not be symmetric, one must do both left and right versions of the arguments therein to verify that the subcategories in question are two-sided ideals.  
%Second, in doing the intermediate steps of the chain of equivalences given therein one may assume without loss that the arbitrary compact objects introduced are in fact simple objects as they are a set of compact generators.  Having done so, the same arguments can be followed, with our \cref{P:twistingandvarieties} used as a replacement for the adjointness property of dual objects.
\end{proof}

Given an $M\in \Stmod (A)$, set 
$$W_{A}(M)=\{P\in X\mid\ \nabla_{P}(M)\neq 0\}.\footnote{Denoted by $\supp_{R} M$ in \cite[Section 5]{BIK}}$$ 
When $M\in \stmod(A)$ one has 
\begin{equation} \label{E:WequalsZforfd}
W_{A}(M)=\ZZ_{A}(M)
\end{equation}
 (cf.\ \cite[Theorem 5.5]{BIK}). 
With our aforementioned setting for $\bar{A}$,  one can analogously define the support $W_{\bar{A}}(M)$ for $M\in\Stmod(\bar{A})$. 
For $M\in \stmod({\bar{A}})$, 
\begin{equation} \label{E:WequalsZforfd-barA}
W_{\bar{A}}(M)=\ZZ_{\bar{A}}(M) 
\end{equation}
(cf.\ \cite[Theorem 5.5]{BIK}). Furthermore, for $M\in \text{mod}(\bar A)$ there is an isomorphism, $\text{Ext}^{\bullet}_{\bar{A}}(M,M)\cong 
\text{Ext}^{\bullet}_{\bar{A}}(M^{*},M^{*})$. It follows by using (\ref{E:ringmap}) that 
\begin{equation} \label{E:varietiesofduals}
\ZZ_{\bar{A}}(M)\cong \ZZ_{\bar{A}}(M^{*})
\end{equation}
(cf.\ \cite[Proposition 5.7.3]{Ben}). 

\subsection{}\label{SS:QCI}  As mentioned in \cref{SS:grb}, the algebra $\bar{A}^{\prime}=\gr u_{\zeta}(\fu)$ is a quantum complete intersection and so admits a rank variety. We adapt rank varieties for quantum complete intersections to the algebra $\bar{A}$ following Benson, Erdmann and Holloway \cite[Sections 4--5]{BEH}.  For details on the various constructions the interested reader 
should refer to \emph{loc.\ cit.} (taking note that our algebra $C^{\prime}$ is denoted ``$A$'' therein).   

One first fixes a block, $e\C E_{\zeta}$, of a finite group $E_{\zeta}$ defined in terms of the various powers of $\zeta$ occuring in the commutation relations for the $X_{\alpha}$  in  \cref{E:relations}, where $e$ is a central idempotent.  As the ground field is the complex numbers the algebra $e\C E_{\zeta}$ is a matrix algebra with a unique simple module $S$.  Let 
$$B=e\C E_{\zeta} \otimes \bar{A}.$$ 
Since  $e\C E_{\zeta}$ is a matrix algebra there is a Morita equivalence between $\Mod (\bar{A})$ and $\Mod (B)$ given by $M \mapsto  S \boxtimes M$, where we write $\boxtimes$ for the outer tensor product of modules. Set $C$ to be the subalgebra of $B$ generated by
\begin{equation}\label{E:Cgenerators}
\left\{  ee_{i}\otimes X_{\gamma_i}, e\otimes K_{j} \mid \ i=1,2,\dots, N,\ j=1,2,\dots,n \right\},
\end{equation}
and let $C'$ denote the subalgebra of $C$ generated by 
\begin{equation}\label{E:Cprimegenerators}
\left\{ee_{i}\otimes X_{\gamma_i} \mid \ i=1,2,\dots, N  \right\},
\end{equation}
where the $e_{i}$ are certain distinguished generators of the finite group $E_{\zeta}$ and $N=|\Phi^{+}|$. Now we have a functor 
\begin{equation}\label{E:Fdef}
F:\Mod (\bar{A}) \to \Mod (C)
\end{equation}
which is the composition of the Morita equivalence with the restriction functor.

One has that $C^{\prime}$ is a normal subalgebra of $C$ with $C//C^{\prime}\cong u_{\zeta}(\ft)$. Moreover, $C^{\prime}$ is naturally a subalgebra of $B$ and is isomorphic to a quantum complete intersection.   The crucial point is that while $\bar{A}'$ is a quantum complete intersection with various powers of $\zeta$ occuring in the commutation relations given in \cref{E:relations}, by construction the analogous relations for $C^{\prime}$ have only $\zeta^{1}$ and $\zeta^{-1}$ appearing.  
As shown in \cite{BEH}, a quantum complete intersection of this type admits a rank variety.

\subsection{}\label{SS:relatingvarieties} Before proceeding we relate cohomology and supports for these algebras. For any quantum complete intersection the cohomology ring is explicitly computed in \cite{BO} and is given by generators and relations with the generators in degrees one and two and where the degree one generators are nilpotent.  If $D$ denotes $\bar{A}'$ or $C'$ and we write $\Ext^{2\bullet}_{D}(\C , \C )$ for the subalgebra of the cohomology ring generated by the generators which lie in degree two, then the inclusion  $\Ext^{2\bullet}_{D}(\C , \C ) \hookrightarrow \Ext^{\bullet}_{D}(\C , \C )$ induces a homeomorphism on the spectrum.  We choose to work with $\Ext^{2\bullet}_{D}(\C , \C )$ and note that it is canonically isomorphic to $ S^{\bullet}({\fu}^{*})$.

In the case when $D$ equals $\bar{A}$ or $C$ the cohomology ring is also generated in degree two and is canonically isomorphic to $ S^{\bullet}({\fu}^{*})$. The proof of this fact uses the same line of reasoning as in the case of $\Ext^{\bullet}_{u_{\zeta}({\mathfrak b})}(\C,\C)$ and, in particular, identifies the cohomology ring of $\bar{A}$ (resp.\ $C)$ with the aforementioned subalgebra of the cohomology ring of $\bar{A}'$ (resp.\ $C'$). 

Consequently, if $D$ denotes any one of these four algebras $\bar A, C, \bar A'$ or $C'$, we have a canonical identification between $\Ext^{2\bullet}_{D}(\C, \C )$ and $S^{\bullet}(\fu^{*})$  and so obtain canonical identifications of the following varieties:
\begin{equation}\label{E:QCIspectrums}
\Spec \left(\Ext_{D}^{2\bullet}(\C , \C ) \right) \cong \Spec \left(S^{\bullet}(\fu^{*}) \right).
\end{equation}
Moreover, for $M$ in $\Stmod (D)$ one can define $\ZZ_{D}(M)$ as in Section~\ref{SS:BIKforAbar}. Using these canonical identifications with subsets of the spectrum of $S^{\bullet}(\fu^{*})$ allows one to sensibly compare these sets for various algebras and their modules.  

Now a direct calculation verifies that $F(\C)= S\boxtimes \C$ is isomorphic to the direct sum of $\dim_{\C} (S)$ copies of $\C$ as a $C$-module. Also, note that the compact objects for $\bar{A}$ and $C$ are generated by the 
collection of one-dimensional simple modules $\{\lambda \mid \lambda\in X_{1}\}$. An important property about the relationship between the algebras $B$ and $C$ is the following fact. 
If $M$ is a module for $B$ then $M$ is a direct summand of  $B\otimes_{C}M$ \cite[Section 3 (Res1)]{BEH}. Furthermore, the induction functor $B\otimes_{C} -$ is exact. A consequence of these facts is that the restriction map
\begin{equation} \label{E:relproj}
\Ext^{\bullet}_{B}(Z_{1},Z_{2})\hookrightarrow\Ext^{\bullet}_{C}(Z_{1},Z_{2})
\end{equation} 
is a monomorphism for any $B$-modules $Z_{1}$ and $Z_{2}$. 

Let $Q_{1}$ be a finite-dimensional $\bar{A}$-module. Combining the observations in the previous paragraph we have the following maps:
\begin{equation} \label{E:maps1}
S^{\bullet}(\fu^{*})\cong \Ext^{2\bullet}_{\bar{A}}(\C ,\C )\rightarrow \Ext^{2\bullet}_{\bar{A}}(Q_{1},Q_{1}) \cong \Ext^{2\bullet}_{B}(S \boxtimes Q_{1} ,S\boxtimes Q_{1} ) 
\end{equation} 
The last isomorphism is given by the Morita equivalence. Since under the Morita equivalence every finite-dimensional $B$-module is isomorphic to $S \boxtimes Q_{1}$ for some finite-dimensional $Q_{1}$, the existence of this homomorphism allows one to define Koszul objects for compact $B$-modules as defined in \cite[Definition 5.10]{BIK}.  These objects will be used below by restricting to $C$ and using \cref{E:relproj}. 

%\hookrightarrow \Ext^{2\bullet}_{C}(S \boxtimes Q_{1} ,S\boxtimes Q_{1})

%The composition of this sequence of maps is an algebra homomorphism whose image
%is $\Ext^{2\bullet}_{C}(\C ,\C )\otimes \C \cdot \operatorname{Id}_{S}$.  

Now let $N$ be a finite-dimensional $C$-module and $M$ be an $\bar{A}$-module. Then under the equivalence of categories: 
$$ 
\Ext^{\bullet}_{C}(N,S\boxtimes M)\cong \Ext^{\bullet}_{B}(B\otimes_{C} N,S\boxtimes M)\cong\Ext^{\bullet}_{\bar{A}}(Q,M)\cong \Ext^{\bullet}_{\Delta}(\C,Q^{*}\otimes M)
$$
for some finite-dimensional $\bar{A}$-module $Q$. This allows one to place an action of $R$ on $\Ext^{\bullet}_{C}(N,S\boxtimes M)$ and define $Z_{C}(S\boxtimes M)$. One can then verify the analogues of the results in \cite[Section 5]{BIK} for the algebra $C$. 

In particular, taking $N=E:=\oplus_{\lambda\in X_{1}}\lambda$, one has $\Ext^{\bullet}_{C}(E,S\boxtimes M)\cong\Ext^{\bullet}_{\bar{A}}(Q,M)$. Therefore, 
$$\ZZ_{C}(S\boxtimes M) \subseteq \ZZ_{\bar{A}}(M).$$
Now for any $M\in \Stmod(\bar{A})$ one has the following composition of maps:
\begin{equation} \label{E:maps2}
\Ext^{\bullet}_{\bar{A}}(E ,M)\xrightarrow{\cong}\Ext^{\bullet}_{B}(S\boxtimes E,S\boxtimes M)\hookrightarrow\Ext^{\bullet}_{C}(S\boxtimes E, S\boxtimes M).
\end{equation}  
Note that $\Ext^{\bullet}_{C}(S\boxtimes E, S\boxtimes M)\cong \Ext^{\bullet}_{C}(E, S\boxtimes M)\otimes S^{*}$.
By localizing at a prime $P\in X$, it follows by definition of support that 
$\ZZ_{\bar{A}}(M)\subseteq \ZZ_{C}(S\boxtimes M)$. Therefore, 
\begin{equation} \label{E:ACvarieties}
\ZZ_{\bar{A}}(M)=\ZZ_{C}(S\boxtimes M)
\end{equation} 
for any $M\in \Stmod(\bar{A})$.  

Next we observe that a similar relationship holds between the supports $W_{\bar{A}}(-)$ and $W_{C}(-)$; namely, for any $M\in\Stmod(\bar A)$,
\begin{equation} \label{E:ACWvarieties}
W_{\bar{A}}(M)=W_{C}(S\boxtimes M).
\end{equation}
In order to see this one can use similar reasoning (via the equivalence of categories) to conclude that for any $P\in X$, 
$$\text{Ext}^{\bullet}_{C}(S\boxtimes E \sslash P, S\boxtimes M)\cong\Ext ^{\bullet}_{\bar{A}}(Q,M)$$ 
for some finite-dimensional $\bar A$-module $Q$. It follows that 
$$\text{min}_{R}\Ext ^{\bullet}_{C}(S\boxtimes E \sslash P, S\boxtimes M)=\text{min}_{R}\Ext ^{\bullet}_{\bar{A}}(Q,M).$$
 The definitions of $\sslash P$ and $\text{min}_{R}$ are given in \cite[Definition 5.10]{BIK} and {\it op.\ cit.} p.\ 589, respectively. Note that for $C$, Koszul objects can be defined 
 for objects of the form $S\boxtimes M$ where $M$ is an $\bar{A}$-module, by using the iterative construction in $\text{mod}(B)$ and then restricting to $C$. 
 Since $S\boxtimes E$ is a compact generator for $C$, it follows that $W_{C}(S\boxtimes M)\subseteq W_{\bar A}(M)$ by \cite[Theorem 5.13]{BIK}. 

On the other hand, let $P\in W_{\bar A}(M)$. Then $\text{Ext}^{\bullet}_{\bar{A}}(E, \nabla_{P}(M))\neq 0$. 
Therefore, by \cite[Proposition 5.2]{BIK}, $\text{Ext}^{\bullet}_{\bar{A}}(E \sslash P, M)_{P}\neq 0$, 
and $\text{Ext}^{\bullet}_{C}(S\boxtimes [E \sslash P], S\boxtimes M)_{P}\neq 0$ (via an injection 
similar to (\ref{E:maps2})). Since by \cref{E:ACvarieties} and 
$$\ZZ_{C}(S\boxtimes [E \sslash P])=\ZZ_{\bar{A}}(E \sslash P)=W_{\bar{A}}(E \sslash P )
=\{P^{\prime}\in X\mid P\subseteq P^{\prime}\},$$
it follows that 
$P\in \operatorname{min}_{R}\Ext^{\bullet}_{C}(S\boxtimes [E \sslash P], S\boxtimes M)$. Consequently, 
$P\in W_{C}(S\boxtimes M)$ by \cite[Theorem 5.13]{BIK}.

\subsection{}\label{SS:DadesLemma} Using the notation established in \cref{SS:QCI}, we introduce rank varieties for $\bar{A}=\gr u_{\zeta}(\fb )$ following \cite{BEH}.   Since $C'$ is a quantum complete intersection for which, by design, only $\zeta$ occurs in the commutation relations, one can define a rank variety as follows.  Fix a field extension, $K$, of $\C$ of sufficiently large transcendence degree (larger than $|\Phi^{+}|$ suffices). If $D$ is one of the algebras we consider, write $D_{K}=K \otimes_{\C}D$.  For brevity set $Y_{\gamma_{i}}=1 \otimes ee_{i}\otimes X_{\gamma_i} \in C'_{K}$, where $ee_{i}\otimes X_{\gamma_i}$ is as in \cref{E:Cprimegenerators}.  Let $V_{C'}^{\rk}(\C)$ denote the subspace of $C'_{K}$ spanned by $\left\{Y_{\gamma_{i}} \mid \gamma_{i} \in \Phi^{+} \right\}$.  Given an element $x=\sum_{i} a_{i}Y_{\gamma_{i}}$ of $V_{C'}^{\rk}(\C)$, write $\langle x \rangle$ for the subalgebra of $C'_{K}$ it generates.

Given a $C'$-module $M$ set 
\begin{equation*}
V_{C'}^{\rk}(M) =  \left\{ x \in V_{C'}^{\rank}(\C )  \mid  K \otimes_{\C} M \text{ is not projective as an $\langle x \rangle$-module} \right\} \cup \left\{0 \right\}.
\end{equation*}  
The field extension $K$ is needed to handle the case when $M$ is infinite-dimensional.  The rank variety does not depend on the choice of $K$ and so we often leave it implicit in what follows.  See \cite[Section 5]{BEH} for further details or \cite[Section 4.5]{BKN} for the analogue in the setting of Lie superalgebras.  An important property of the rank variety is that it satisfies Dade's Lemma; that is, $M$ is a projective $C'$-module if and only if $V_{C'}^{\rk}(M)=\{0 \}$ (see \cite[Theorem 5.4]{BEH}). 

A $C$-module $M$ is a $C'$-module by restriction and so one can define the rank variety 
\[
V_{C}^{\rk }(M)=V_{C'}^{\rk}(M).
\]
Using the fact that $C'$ is a normal subalgebra of $C$, it follows from a spectral sequence argument that a $C$-module $M$ is projective if and only if it is projective upon restriction to $C'$. Hence, the rank variety of $C$ also satisfies Dade's Lemma. % (e.g.\ see \cite[Section 5.3]{GK})

Recall from \cref{SS:relatingvarieties} that $S^{\bullet}(\fu^{*})$ canonically identifies with a subalgebra of $\Ext_{C'}^{\bullet}(\C, \C)$ and so acts on $\Ext_{C'}^{\bullet}(\C , N)$ for any compact $C'$-module $N$. In particular, one can define $Z_{C'}(N)$ for any such $N$. By \cite[Proposition 5.1]{BEH} there is a bijection, 
\[
\beta^{*}: \ZZ_{C'}(\C ) \to V_{C'}^{\rk}(\C ),
\] given by taking an irreducible subvariety to the point of $V_{C'}^{\rk}(\C)$ which corresponds to the unique generic point of the subvariety. The results of \cref{SS:relatingvarieties} imply that there is a canonical identification of $\ZZ_{C'}(\C)$ and $\ZZ_{C}(\C)$ and so we obtain a bijection 
\begin{equation}\label{E:betastar}
\beta^{*}: \ZZ_{C}(\C ) \to V_{C}^{\rk}(\C ).
\end{equation}

Given an $\bar{A}$-module, $M$, we set
\begin{equation} \label{E:barArankvariety}
V_{\bar{A}}^{\rk }(M)=V_{C}^{\rk }(F(M)),
\end{equation}  where $F$ is the functor given in \cref{E:Fdef}. We also note that combining  \cref{E:ACvarieties} and \cref{E:betastar} provides a bijection 
\begin{equation} \label{E:betastarbarA}
\beta^{*}: W_{\bar{A}}(\C ) = \ZZ_{\bar{A}}(\C) \to V_{\bar{A}}^{\rk}(\C).
\end{equation} 

\iffalse

Recall from \cref{L:diagonalmap} that $\gr \Delta: \bar{A} \to \bar{A}\otimes \bar{A}$ induces the diagonal map $\ZZ_{\bar{A}}(\C) \to \ZZ_{\bar{A}}(\C) \times \ZZ_{\bar{A}}(\C)$. This implies that the induced map (via $\beta^{*}$), 
\[
V_{\bar{A}}^{\rk}(\C ) \to V_{\bar{A} \otimes \bar{A}}^{\rk}(\C ) \cong V_{\bar{A}}^{\rk }(\C ) \times V_{\bar{A}}^{\rk  }(\C ),
\]
 is also the diagonal map.
\fi

\section{Comparison of supports for \texorpdfstring{$A$ and $\bar{A}$}{A and gr(A)}} \label{S:comparison}

\subsection{} We will first consider the stable module category for $C$. Let ${\mathcal V}$ be a specialization closed set in $X$, 
and let $L_{\mathcal V}$ and $\Gamma_{\mathcal V}$ be the (co)localization functors as defined in \cite[Definition 4.6]{BIK}. If $M\in \text{Stmod}(C)$ then 
by \cite[Theorem 5.6]{BIK}
\begin{equation}\label{eq:Wsupportincl1}
W_{C}(\Gamma_{\mathcal V}(M))={\mathcal V}\cap W_{C}(M)
\end{equation} 
\begin{equation}\label{eq:Wsupportincl2}
W_{C}(L_{\mathcal V}(M))=(X -{\mathcal V})\cap W_{C}(M).
\end{equation}

Let $P\in X $ and set 
$${\mathcal V}(P)=\{P^{\prime}\in X \mid P\subseteq P^{\prime}\},$$ 
$${\mathcal Z}(P)=\{P^{\prime}\in X \mid P^{\prime}\nsubseteq P \},\quad \text{so that}$$
$$X -{\mathcal Z}(P)=\{P^{\prime}\in X \mid P^{\prime} \subseteq P\}.$$ 

The points $x$ in $V_{C}^{\rank}(\C)$ correspond to primes in $\Proj(S^{\bullet}({\mathfrak u}^{*}))$. 
Denote this correspondence by $\beta(x)={P}_{x}$, with $\beta^{*}=\beta^{-1}$. 

The following proposition shows that $V_{C}^{\rank}(-)$ shares inclusion properties analogous to one direction of the equalities \cref{eq:Wsupportincl1,eq:Wsupportincl2} for $W_{C}(-)$.  

\begin{prop}\label{P:rank inclusion} Let ${\mathcal V}$ be a specialization closed set in $X$. If $M\in \operatorname{Stmod}(C)$ then

\begin{itemize} 
\item[(a)] $V_{C}^{\operatorname{rank}}(L_{\mathcal V}(M))\subseteq \beta^{*}(X-{\mathcal V})\cap V^{\rank}_{C}(M).$
\item[(b)] $V_{C}^{\operatorname{rank}}(\Gamma_{\mathcal V}(M))\subseteq \beta^{*}({\mathcal V})\cap V_{C}^{\rank}(M)$
\item[(c)] $V_{C}^{\operatorname{rank}}(\nabla_{P}(M))\subseteq \beta^{*}(\{P\})\cap V^{\rank}_{C}(M).$
\end{itemize}
\end{prop} 

\begin{proof} Recall $E=\oplus_{\lambda \in X_{1}} \lambda$. Let $\zeta\in K^{n}$ and 
$$\text{res}:\text{Ext}^{\bullet}_{C_{K}}(E, M) \rightarrow\Ext ^{\bullet}_{\langle \zeta \rangle}(E, M)$$
be the restriction map with $\text{Ext}^{\bullet}_{\langle \zeta \rangle}(E, M)\cong\Ext ^{\bullet}_{C_{K}}(\text{coind}_{\langle \zeta \rangle}^{C_{K}} E, M)$. 
This induces an injection of supports  
$$\ZZ_{\langle \zeta \rangle}(E, M)\hookrightarrow \ZZ_{C_{K}}(E, M).$$ 

Next we claim: 
\begin{equation} \label{E:RankInclusionClaim}
\text{if } \zeta\in V_{\langle \zeta \rangle}^{\text{rank}}(M)\subseteq V_{C}^{\text{rank}}(M) \text{ then }\Ext ^{\bullet}_{C}(E, M)_{P_{\zeta}}\neq 0.
\end{equation}
One has an inclusion of rings $S^{\bullet}({\mathfrak u}_{\mathbb C})\subseteq S^{\bullet}({\mathfrak u}_{K})$ which induces a map 
$\Spec(S^{\bullet}({\mathfrak u}_{K}))\rightarrow \Spec(S^{\bullet}({\mathfrak u}_{\mathbb C}))$ defined by $P\rightarrow P\cap S^{\bullet}({\mathfrak u}_{\mathbb C})$. 
By the construction of the generic point (see \cite[Sections 2 and 3]{BCR} for details and notation; here again $N=|\Phi^{+}|$), 
$$(x_{1}-\zeta_{1},\dots, x_{N}-\zeta_{N})\cap S^{\bullet}({\mathfrak u}_{\mathbb C}) =P_{\zeta}.$$ 
Now suppose that $\text{Ext}^{\bullet}_{C}(E, M)_{P_{\zeta}}=0$. Since $\zeta\in V_{\langle \zeta \rangle}^{\text{rank}}(M)$, it follows that 
$$\text{Ext}^{\bullet}_{C_{K}}(E, M)_{(x_{1}-\zeta_{1},\dots,x_{N}-\zeta_{N})}\neq 0.$$ 
But, 
$\text{Ext}^{\bullet}_{C_{K}}(E, M)\cong K\otimes\Ext ^{\bullet}_{C}(E , M)$. 
Let $\alpha\otimes m\in K\otimes\Ext ^{\bullet}_{C}(E, M)$. There exists $t\in S^{\bullet}({\mathfrak u}_{\mathbb C})-P_{\zeta}$ 
with $(1\otimes t).(\alpha\otimes m)=0$. Now if $1\otimes t\in (x_{1}-\zeta_{1},\dots, x_{N}-\zeta_{N})$ then $t\in (x_{1}-\zeta_{1},\dots, x_{N}-\zeta_{N})\cap 
S^{\bullet}({\mathfrak u}_{\mathbb C})=P_{\zeta}$, which cannot occur. This shows that $\text{Ext}^{\bullet}_{C_{K}}(E, M)_{(x_{1}-\zeta_{1},\dots,x_{N}-\zeta_{N})}=0$,
which is a contradiction. The claim at the beginning of the paragraph now follows.  

The proofs of parts (a) and (b) follow in four steps. 
\vskip .25cm 

\noindent 
(i) First, we show that  $V_{C}^{\operatorname{rank}}(L_{\mathcal V}(M))\subseteq \beta^{*}(X -{\mathcal V})$. 
Suppose that $\zeta\in V_{C}^{\operatorname{rank}}(L_{\mathcal V}(M))$; then by \cref{E:RankInclusionClaim},
$\text{Ext}^{\bullet}_{C}(E, L_{\mathcal V}(M))_{P_{\zeta}}\neq 0$. 
From \cite[Theorem 4.7]{BIK}, it follows that  
$$\text{Ext}^{\bullet}_{C}(E, L_{{\mathcal Z}(P_{\zeta})}L_{\mathcal V}(M))\neq 0.$$ 
Thus, $L_{{\mathcal Z}(P_{\zeta})}L_{\mathcal V}(M)\neq 0$. 

Using \cref{eq:Wsupportincl2} we have 
$$W_{C}(L_{{\mathcal Z}(P_{\zeta})}L_{\mathcal V}(M))\subseteq (X -{\mathcal Z}(P_{\zeta}))\cap (X -{\mathcal V}) \cap W_{C}(M).$$
Suppose that $P_{\zeta}\in {\mathcal V}$. If $P\in  (X -{\mathcal Z}(P_{\zeta}))\cap (X -{\mathcal V}) $ then $P_{\zeta}\subseteq P$. 
But ${\mathcal V}$ is specialization closed so $P\in {\mathcal V}$, which is a contradiction. Therefore, 
$W_{C}(L_{{\mathcal Z}(P_{\zeta})}L_{\mathcal V}(M))=\varnothing$; that is, $L_{{\mathcal Z}(P_{\zeta})}(L_{\mathcal V}(M))=0$, which also cannot occur. 
We can now conclude that $\zeta\in \beta^{*}(X -{\mathcal V})$. 
\vskip .25cm

\noindent 
(ii) Second, $V_{C}^{\operatorname{rank}}(\Gamma_{\mathcal V}(M))\subseteq \beta^{*}({\mathcal V})$. 
Suppose that $\zeta\in V_{C}^{\operatorname{rank}}(\Gamma_{\mathcal V}(M))$; then by \cref{E:RankInclusionClaim}, $P_{\zeta}\in Z_{C}(E, 
\Gamma_{\mathcal V}(M))$. It follows by (\ref{eq:Wsupportincl1}) and \cite[Theorem 5.15]{BIK} that $P_{\zeta}\in {\mathcal V}=\text{cl}({\mathcal V})$. Consequently, 
$\zeta\in \beta^{*}({\mathcal V})$. 
\vskip .25cm 

\noindent 
(iii) Third, $V_{C}^{\operatorname{rank}}(\Gamma_{\mathcal V}(M))\subseteq  V^{\text{rank}}_{C}(M)$. Suppose that 
$\zeta\in V_{C_{K}}^{\operatorname{rank}}(\Gamma_{\mathcal V}(M))$. From (ii), $\zeta\in \beta^{*}({\mathcal V})$. One has the distinguished triangle 
\begin{equation}\label{E:distinguished} 
\rightarrow L_{\mathcal V}(M)\rightarrow M \rightarrow \Gamma_{\mathcal V}(M) \rightarrow 
\end{equation} 
From (i), $L_{\mathcal V}(M)|_{\langle \zeta \rangle}$ is projective. Hence $M|_{\langle \zeta \rangle} 
\cong \Gamma_{\mathcal V}(M)|_{\langle \zeta \rangle}$. Therefore, by using the definition of the rank variety, one 
has $\zeta\in V^{\text{rank}}_{C_{K}}(M)$. Now (b) follows from (ii) and (iii).

\vskip .25cm
\noindent
(iv) Finally, we show that $V_{C}^{\operatorname{rank}}(L_{\mathcal V}(M))\subseteq V^{\text{rank}}_{C}(M)$. Using the distinguished triangle (\ref{E:distinguished}) and (iii), 
one has 
$$V_{C}^{\operatorname{rank}}(L_{\mathcal V}(M))\subseteq V_{C}^{\operatorname{rank}}(\Gamma_{\mathcal V}(M)) \cup 
V_{C}^{\operatorname{rank}}(M) \subseteq  V_{C}^{\operatorname{rank}}(M).$$
Together with (i) this proves (a).
\vskip .25cm 

(c) From parts (a) and (b), one has 
\begin{align*} 
V_{C}^{\text{rank}}(\nabla_{P}(M))&\subseteq  \beta^{*}({\mathcal V}(P))\cap V_{C}^{\text{rank}}(L_{{\mathcal Z}(P)}(M)) \\ 
&\subseteq  \beta^{*}({\mathcal V}(P))\cap \beta^{*}(X -{\mathcal Z}(P))\cap V_{C}^{\text{rank}}(M) \\
&\subseteq  \beta^{*}(\{P\})\cap V_{C}^{\text{rank}}(M). \qedhere
\end{align*}
\end{proof} 

\subsection{Tensor Product Theorems} For our purposes we will need versions of tensor product theorems for both $A$ and $\bar{A}$. We begin with $\bar A$.

\begin{theorem}\label{T:grWintersection} Let $Q_{1} \in \operatorname{stmod}(\bar{A})$ and $Q_{2} \in \operatorname{Stmod}(\bar{A})$. If $Q_{1}\otimes Q_{2}=0$ in $\operatorname{Stmod}(\Delta)$ then 
${W}_{\bar{A}}(Q_{1})\cap {W}_{\bar{A}}(Q_{2})=\varnothing$. 
\end{theorem} 

\begin{proof} 
We will argue by contradiction. Suppose that $P\in {W}_{\bar{A}}(Q_{1})\cap {W}_{\bar{A}}(Q_{2})$. Then 
by (\ref{E:ACWvarieties}), 
$$P\in {W}_{C}(S\boxtimes Q_{1})\cap {W}_{C}(S\boxtimes Q_{2}).$$ 
It follows that for $j=1,2$ one has $\nabla_{P}(S\boxtimes Q_{j})\neq 0$. Since Dade's Lemma is true for $C$ (see \cref{SS:DadesLemma}), $V_{C}^\text{rank}(\nabla_{P}(S\boxtimes Q_{j}))\neq \varnothing$.

From \cref{P:rank inclusion}(c), one has $V_{C}^\text{rank}(\nabla_{P}(S\boxtimes Q_{j}))= \{x \}$, where $\beta^{*}( P)=\{x\}$. 
Combining this with the previous statement shows that 
$$V_{C}^\text{rank}(\nabla_{P}(S\boxtimes Q_{1}))\cap V_{C}^\text{rank}(\nabla_{P}(S\boxtimes Q_{2}))=\{x\}.$$  
Also by \cref{P:rank inclusion}(c) one has that $V_{C}^\text{rank}(\nabla_{P}(S\boxtimes Q_{j}))\subseteq V_{C}^\text{rank}(S\boxtimes Q_{j})$. Consequently, $x\in {V}_{C}^{\text{rank}}(S\boxtimes Q_{1})\cap {V}_{C}^{\text{rank}}(S\boxtimes Q_{2})$.  

Suppose that $\text{Ext}^{1}_{\langle x \rangle}(S\boxtimes [Q_{1}^{*}\otimes  E],S\boxtimes Q_{2})=0$, where, as before, $E=\oplus_{\lambda\in X_{1}}\lambda$. Then 
the summand 
\begin{equation} \label{Ext1Q1Q2} 
\text{Ext}^{1}_{\langle x \rangle}(S\boxtimes Q_{1}^{*},S\boxtimes Q_{2})=0.  
\end{equation} 
By (\ref{E:varietiesofduals}), ${Z}_{\bar{A}}(Q_{1})={Z}_{\bar{A}}(Q_{1}^{*})$. 
According to \cref{E:ACvarieties} and the main result of \cite{BeEr:11}, for a finite-dimensional $\bar{A}$-module $M$, 
$${Z}_{\bar{A}}(M)={Z}_{C}(S\boxtimes M)\cong V^{\text{rank}}_{C}(S\boxtimes M).$$
It now follows that 
$$V_{C}^{\rank}(S\boxtimes Q_{1})=V_{C}^{\rank}(S\boxtimes Q_{1}^{*}).$$ 
This shows that $S\boxtimes Q_{1}$ is projective over $\langle x \rangle$ if and only if $S\boxtimes Q_{1}^{*}$ is projective over $\langle x \rangle$. 
By assumption, $S\boxtimes Q_{1}$ is not projective over $\langle x \rangle$ so $S\boxtimes Q_{1}^{*}$ is not projective over 
$\langle x \rangle$. Therefore, there exists a non-projective $\langle x \rangle$-summand $Z_{1}$ of $S\boxtimes Q_{1}^{*}$. Similarly, $S\boxtimes Q_{2}$ has a non-projective $\langle x \rangle$-summand $Z_{2}$. 
From \cref{Ext1Q1Q2}, $\text{Ext}^{1}_{\langle x \rangle}(Z_{1},Z_{2})=0$. This is a contradiction, 
and hence
\begin{equation} \label{E:nonzeroExtoverx}
\text{Ext}^{1}_{\langle x \rangle}(S\boxtimes [Q_{1}^{*}\otimes  E],S\boxtimes Q_{2})\neq 0. 
\end{equation}

By the finite-dimensionality of $Q_{1}$ and Theorem~\ref{thm:injectivemap},
$$\text{Ext}^{\bullet}_{\bar{A}}(Q_{1}^{*},Q_{2})\cong\Ext ^{\bullet}_{\Delta}({\mathbb C},Q_{1}\otimes Q_{2}).$$ 
Moreover, from the equivalence of categories, 
$$\text{Ext}^{\bullet}_{\bar{A}}(Q_{1}^{*},Q_{2}) \cong\Ext ^{\bullet}_{B}(S\boxtimes Q_{1}^{*}, S\boxtimes Q_{2}).$$ 
Therefore, if $Q_{1}\otimes Q_{2}=0$ then $\text{Ext}^{\bullet}_{B}(S\boxtimes Q_{1}^{*}, S\boxtimes Q_{2})=0$.  

Let $Q$ be a finite-dimensional $\bar{A}$-module. Then 
$\text{Ext}^{\bullet}_{B}(S\boxtimes Q^{*},S\boxtimes Q^{*})$ acts via Yoneda product on $\text{Ext}^{\bullet}_{B}(S\boxtimes Q^{*},S\boxtimes Q_{2})$. 
Now by using the equivalence and Theorem~\ref{thm:injectivemap}(b), one has ring isomorphisms
$$\text{Ext}^{\bullet}_{B}(S\boxtimes Q^{*},S\boxtimes Q^{*})\cong\text{Ext}^{\bullet}_{\bar{A}}(Q^{*},Q^{*})\cong\Ext ^{\bullet}_{\Delta}({\mathbb C},Q\otimes Q^{*}).$$ 
Moreover, there is a ring homomorphism, 
$$S^{\bullet}({\mathfrak u}^{*})\cong\Ext ^{\bullet}_{\Delta}({\mathbb C},{\mathbb C}) \rightarrow\Ext ^{\bullet}_{\Delta}({\mathbb C},Q\otimes Q^{*}).$$ 
Combining these facts, one has a ring homomorphism, 
$$S^{\bullet}({\mathfrak u}^{*})\rightarrow\Ext ^{\bullet}_{B}(S\boxtimes Q^{*},S\boxtimes Q^{*})$$
which induces an action of $S({\mathfrak u}^{*})$ on $\text{Ext}^{\bullet}_{B}(S \boxtimes Q^{*},S \boxtimes Q_{2})$. 
One also has the following commutative diagram that involves restriction of cohomology. 
$$
\CD
\Ext^{\bullet}_{B}(S\boxtimes Q^{*},S\boxtimes Q^{*}) @>\text{res}>>\Ext ^{\bullet}_{\langle x \rangle}(S\boxtimes Q^{*},S\boxtimes Q^{*}) \\
@AAA @AAA           \\
S^{\bullet}({\mathfrak u}^{*}) @>\text{res}>> S^{\bullet}(\langle x \rangle ^{*})
\endCD
$$
%\cong \text{Ext}^{\bullet}_{\Delta}({\mathbb C},(Q^{*})^{*}\otimes Q^{*})

This setup can be now used when $Q=Q_{1}^{*}$ (resp.\ $Q=\oplus_{\lambda\in X_{1}}\lambda$) to define the topological spaces in the 
commutative diagram below. 

$$
\CD
{Z}_{\langle x \rangle }(S\boxtimes Q_{1}^{*}, S\boxtimes Q_{2}) @>\text{res}^{*} >> {Z}_{B_{K}}(S\boxtimes Q_{1}^{*}, S\boxtimes Q_{2}) \\
@VVV @VVV           \\
{Z}_{\langle x \rangle }(\oplus_{\lambda\in X_{1}} S\boxtimes \lambda, S\boxtimes Q_{2}) @>\text{res}^{*} >> {Z}_{B_{K}}(\oplus_{\lambda\in X_{1}} S\boxtimes \lambda, S\boxtimes Q_{2}) 
\endCD
$$
The vertical maps are inclusions since $S\boxtimes Q_{1}^{*}$ has a composition series with sections of the form $S\boxtimes \sigma$ where $\sigma \in X_{1}$. 
The bottom horizontal map 
$$\text{res}^{*}:{Z}_{\langle x \rangle }(\oplus_{\lambda\in X_{1}} S\boxtimes \lambda, S\boxtimes Q_{2}) \hookrightarrow {Z}_{B_{K}}(\oplus_{\lambda\in X_{1}} S\boxtimes \lambda, S\boxtimes Q_{2})$$
is an inclusion because 
$$\text{Ext}^{\bullet}_{\langle x \rangle }(\oplus_{\lambda\in X_{1}} S\boxtimes \lambda, S\boxtimes Q_{2})\cong\Ext ^{\bullet}_{B_{K}}(\text{coind}_{\langle x \rangle}^{B_{K}}(\oplus_{\lambda\in X_{1}} S\boxtimes \lambda), S\boxtimes Q_{2}),$$ 
and $\text{coind}_{\langle x \rangle}^{B_{K}}(\oplus_{\lambda\in X_{1}} S\boxtimes \lambda)$ has a composition series with sections of the form $S\boxtimes \sigma$ where $\sigma\in X_{1}$.  Hence the top map $\text{res}^{*}$ is also an inclusion. But
since $\text{Ext}^{\bullet}_{B}(S\boxtimes Q_{1}^{*}, S\boxtimes Q_{2})=0$, it follows that ${Z}_{B_{K}}(S\boxtimes Q_{1}^{*}, S\boxtimes Q_{2})=\varnothing.$
This yields a contradiction because from \cref{E:nonzeroExtoverx}, ${Z}_{\langle x \rangle }(S\boxtimes Q_{1}^{*}, S\boxtimes Q_{2})\neq \varnothing$. \qedhere

\end{proof}

\subsection{} We also need a version of the tensor product theorem for $A$. The following theorem generalizes the result stated in \cite[Theorem 2.5]{FeW} to infinitely generated modules. 
Recall that for $\zeta\in \operatorname{H}^{n}(A,{\mathbb C})$, let $L_{\zeta}$ be the kernel of the map $\zeta:\Omega^{n}({\mathbb C})\rightarrow {\mathbb C}$ (i.e., the Carlson module). Note that $L_{\zeta}$ is a finite dimensional $A$-module.

\begin{theorem}\label{T:Lzetatensor} Let $M\in \operatorname{Stmod}(A)$, and $\zeta, \zeta_{i}\in R$ be homogeneous elements of positive degree ($1\le i\le t)$. 
\begin{itemize} 
\item[(a)] $W_{A}(L_{\zeta}\otimes M)=W_{A}(L_{\zeta})\cap W_{A}(M)$.
\item[(b)]  $W_{A}(\otimes_{i=1}^{t} L_{\zeta_{t}}\otimes M)=[\cap_{i=1}^{t} W_{A}(L_{\zeta})]\cap W_{A}(M)$. 
\end{itemize}
\end{theorem}

\begin{proof} The main verification involves (a) since (b) follows from part (a) by induction. 
One has 
$$W_{A}(L_{\zeta}\otimes M)\subseteq W_{A}(L_{\zeta})\cap W_{A}(M).$$ 
So it remains to prove the other inclusion. 
Note that 
\begin{equation} \label{E:WAforLzeta}
W_{A}(L_{\zeta})=\{P\in X \mid \zeta\in P\} \linebreak[0] = \mathcal{V}(\langle \zeta \rangle)
\end{equation}
by \cite[Theorem 2.5]{FeW}. Let 
$P\in W_{A}(L_{\zeta})\cap W_{A}(M)$, and assume that 
$P\notin W_{A}(L_{\zeta}\otimes M)$. We will argue by contradiction. 

Let  $E=\oplus_{\lambda \in X_{1}}\lambda$ be the direct sum of simple $A$-modules. 
Since $P\notin W_{A}(L_{\zeta}\otimes M)$, $\nabla_{P}(L_{\zeta}\otimes M)=0$, thus 
$\text{Ext}^{\bullet}_{A}(E,\nabla_{P}(L_{\zeta}\otimes M))=0$. From \cite[Proposition 5.2]{BIK}, it follows that 
$$\text{Ext}^{\bullet}_{A}(E \sslash P,\nabla_{P}(L_{\zeta}\otimes M))_{P}=0.$$

Consider the short exact sequence $0\rightarrow L_{\zeta}\rightarrow \Omega^{n}(\C) \rightarrow \C \rightarrow 0$ represented by $\zeta$. 
One can tensor this sequence with $M$ and apply the argument in the proof of \cite[Theorem 2.5]{FeW} to obtain 
$\text{Ext}_{A}^{j}(E\sslash P,M)_{P}=\zeta.\text{Ext}_{A}^{j}(E\sslash P,M)_{P}$ 
for all $j$. 
Now using the facts that $\text{Ext}_{A}^{\bullet}(E\sslash P,M)$ is $P$-torsion and $\zeta\in P$ (cf.\ \cite[Lemma 5.11]{BIK}) shows that 
$$\text{Ext}_{A}^{\bullet}(E\sslash P,M)_{P}=0.$$
Therefore, $\text{Ext}_{A}^{\bullet}(E,\nabla_{P}(M))=0$, thus $P\notin W_{A}(M)$, which is a contradiction. Consequently, 
$P\in W_{A}(L_{\zeta}\otimes M)$. 
\end{proof}

\subsection{Realizing Supports via Compact Objects in \texorpdfstring{$A$ and $\bar{A}$}{A and gr(A)}}  We will use the following statement several times: if $M\in \text{Stmod}(A)$ and $P\in X$ then
\begin{equation}\label{E:Wgrnablainclusion} 
W_{\bar A}(\gr \nabla_{P}(M))\subseteq \{P^{\prime}\in X\mid P^{\prime}\subseteq P\} = X - {\mathcal Z}(P). 
\end{equation} 

To verify this claim, recall $E=\oplus_{\lambda\in X_{1}}\lambda$ so that $W_{A}(E\sslash P^{\prime})={\mathcal V}(P^{\prime})$. 
Let $P^{\prime}\in W_{\bar A}(\gr \nabla_{P}(M))$, and suppose that $P^{\prime}$ is not a subset of $P$. 
Then $P\notin {\mathcal V}(P^{\prime})$, and 
$$E\sslash P^{\prime}\otimes \nabla_{P}(M)=0.$$ Therefore, 
$\gr E\sslash P^{\prime} \otimes \nabla_{P}(M)=0$, and by Theorem~\ref{T:grWintersection},
$$W_{\bar{A}}(\gr E\sslash P^{\prime})\cap W_{\bar{A}}(\gr \nabla_{P}(M))=\varnothing.$$
This is a contradiction since $P^{\prime}\in W_{A}(E\sslash P^{\prime})\subseteq W_{\bar{A}}(\gr E\sslash P^{\prime})$. 
Consequently, $P^{\prime}\subseteq P$. 

\begin{prop} \label{P:realization} Let $P\in X$. Then there exists $Y$ in $\operatorname{stmod}(A)$ such that 
$$W_{A}(Y)={\mathcal V}(P)=W_{\bar{A}}(\gr Y).$$ 
\end{prop} 

\begin{proof} First we will prove that $P\in W_{\bar{A}}(\gr \nabla_{P}({\mathbb C}))$ for all $P\in X $. By \cref{E:Wgrnablainclusion},
$$W_{\bar A}(\gr \nabla_{P}({\mathbb C}))\subseteq \{P^{\prime}\mid P^{\prime}\subseteq P\}.$$ 
Using the fact that $R$ is Noetherian, we can perform ``induction on primes.'' If $P$ is a minimal prime then $W_{\bar{A}}(\gr \nabla_{P}({\mathbb C}))=\{P\}$ because $\gr \nabla_{P}({\mathbb C})$
is not projective by Proposition~\ref{P:btou}. 

Now assume that $P'\in W_{\bar{A}}(\gr \nabla_{P'}({\mathbb C}))$ for all $P'\subsetneq P$, and fix such a $P'$. Using \eqref{E:WAforLzeta} and \cite[Theorem 2.5]{FeW}, let $Y=\otimes_{\zeta}L_{\zeta}$ be such that $W_{A}(Y)={\mathcal V}(P)$. Using the hypotheses and \cref{T:NablaAsTensor}, one has 
$Y \otimes \nabla_{P'}({\mathbb C})=0$. Therefore, $\gr Y \otimes\nabla_{P'}({\mathbb C})=0$. By Theorem~\ref{T:grWintersection}, it follows that 
$W_{\bar{A}}(\gr Y) \cap W_{\bar{A}}(\gr \nabla_{P'}({\mathbb C})) =\varnothing$. This implies that $P'\notin W_{\bar{A}}(\gr Y)$. 

Since by \cite[Theorem 5.1.5]{BIK}, $P\in W_{A}(\nabla_{P}(Y))=\{P\}\cap W_{A}(Y)$, one has $\nabla_{P}(Y)\cong Y\otimes \nabla_{P}({\mathbb C})\neq 0$. Therefore, 
$\gr Y\otimes \nabla_{P}({\mathbb C})\neq 0$, and (using \cite[Theorem 5.2]{BIK})
$$\varnothing \neq W_{\bar{A}}(\gr Y \otimes\nabla_{P}(\C)) \subseteq W_{\bar{A}}(\gr Y)\cap W_{\bar{A}}(\gr \nabla_{P}({\mathbb C}))\subseteq \{P^{\prime}\in X \mid P^{\prime}\subseteq P\}.$$ 
In the previous paragraph, we showed that $\{P^{\prime}\in X \mid P^{\prime}\subsetneq P\}\cap W_{\bar{A}}(\gr Y)=\varnothing$. Consequently, 
$P\in W_{\bar{A}}(\gr \nabla_{P}({\mathbb C}))$. 

In order to prove the statement of the proposition, note that for $Y=\otimes_{\zeta}L_{\zeta}$: (i) $W_{A}(Y)={\mathcal V}(P)$, and (by \cref{E:AvarietyinbarAvariety} and \cref{E:WequalsZforfd})
(ii) $W_{A}(Y)\subseteq W_{\bar{A}}(\gr Y)$. Suppose that $P'\in  W_{\bar{A}}(\gr Y)$ with $P'\notin{\mathcal V}(P)$. 
Then 
$$W_{A}(Y\otimes\nabla_{P'}({\mathbb C}))=\{P'\}\cap {\mathcal V}(P)=\varnothing.$$ 
Using \cref{T:grWintersection}, this implies that 
$$W_{\bar{A}}(\gr Y) \cap W_{\bar{A}}(\gr \nabla_{P'}({\mathbb C}))=\varnothing.$$ 
This contradicts the fact that $P'\in W_{\bar{A}}(\gr Y) \cap W_{\bar{A}}(\gr \nabla_{P'}({\mathbb C}))$. Hence, 
$W_{\bar{A}}(\gr Y)={\mathcal V}(P)$.
\end{proof}

\subsection{} The following result will be needed in the verification of Assumption \ref{A:projectivity} in Theorem~\ref{P:projectivityverify}.

\begin{theorem} \label{T:Winclusion} Let $Q\in \operatorname{Stmod}(A)$. Then $W_{A}(Q)\subseteq W_{\bar A}(\gr Q)$. 
\end{theorem} 

\begin{proof} Let $P\in W_{A}(Q)$. Using \cref{P:realization} choose $Y=\otimes_{j=1}^{t}L_{\zeta_{j}}$ in $\text{stmod}(A)$ such that $W_{A}(Y)={\mathcal V}(P )=W_{\bar{A}}(\gr Y)$. 
First observe that by \cref{T:Lzetatensor,T:NablaAsTensor},
$$W_{A}(\nabla_{P}(Y\otimes Q))=W_{A}(Y\otimes \nabla_{P}(Q))=W_{A}(Y)\cap W_{A}(\nabla_{P}(Q)).$$ 
Since $P\in W_{A}(Q)$ and $W_{A}(Y)={\mathcal V}(P)$, one has 
\begin{equation}
W_{A}(\nabla_{P}(Y\otimes Q))=\{P\}.
\end{equation}

On the other hand, by analyzing the graded module and using \cref{E:Wgrnablainclusion},
\begin{eqnarray*} 
W_{\bar{A}}(\gr \nabla_{P}(Y\otimes Q))&=& W_{\bar{A}}(\gr Y\otimes \nabla_{P}(Q))\\
&\subseteq & W_{\bar{A}}(\gr Y)\cap W_{\bar{A}}(\gr  \nabla_{P}(Q))\\ 
&\subseteq& {\mathcal V}(P)\cap \{P^{\prime}\in X  \mid P^{\prime}\subseteq P \} \\
&=& \{ P \}.
\end{eqnarray*} 
Since $W_{A}(\nabla_{P}(Y\otimes Q))\neq \varnothing$, it follows that $W_{\bar{A}}(\gr \nabla_{P}(Y\otimes Q))\neq \varnothing$ by Proposition~\ref{P:btou} and 
\cite[Theorem 5.2]{BIK}. From the inclusions above, this implies that $W_{\bar{A}}(\gr Y)\cap W_{\bar{A}}(\gr  \nabla_{P}(Q))\linebreak[0]=\{P\}$, thus 
$P\in W_{\bar{A}}(\gr  \nabla_{P}(Q))$. Therefore, $P\in W_{\bar A}(\gr Q)$ because 
\begin{equation*}
W_{\bar{A}}(\gr  \nabla_{P}(Q))\subseteq W_{\bar A}(\gr Q)\cap W_{\bar{A}}(\gr \nabla_{P}({\mathbb C}))\subseteq W_{\bar A}(\gr Q). \qedhere
\end{equation*}
\end{proof}

\section{Classification of Tensor Ideals for Quantum Groups} 

\subsection{Construction of the Support Data}  \label{SS:ConstructionOfSupport}
Let $\bK=\Stmod(U_{\zeta}(\fg))$, $\bK^{c}=\stmod(U_{\zeta}(\fg))$ with $\ell>h$. By work of Ginzburg and Kumar \cite[Main Theorem]{GK}, we have 
$$R:=\HH^{2\bullet}(u_{\zeta}(\fg),\C )\cong \C [\NN]$$ and 
$\HH^{2\bullet +1}(u_{\zeta}(\fg),\C )=0$. Let $X=\Proj(\C [\NN])$,
$X_{G}=G\text{-}\Proj(\C [\NN])$, and let $\rho: X \to X_{G}$ be as in \cref{SS:zariski}. Let  
${\mathcal X}_{cl}$ be the set of all closed sets in $X_{G}$. We will use the cohomology of the small quantum group $u_{\zeta}(\fg)$ to construct a support data 
from $\bK^{c}$ to ${\mathcal X}_{cl}$. 

For any $M\in \mod(u_{\zeta}(\fg))$, set 
$${V}_{u_{\zeta}(\fg)}(M)=\{P\in X\mid \Ext^{\bullet}_{u_{\zeta}(\fg)}(M,M)_{P}\neq 0 \} .$$ 
%For $M\in \bK^{c}$, let $M|_{u_{\zeta}(\fg)}$ be $M$ regarded as a $u_{\zeta}(\fg)$-module. 
Let $V:\bK^{c}\rightarrow {\mathcal X}_{cl}$ 
be defined as 
\begin{equation} \label{e:supportdatadef} 
V(M):=\rho \left({V}_{u_{\zeta}(\fg)}(M) \right), %=\{P\in \Proj(\C [\NN])\mid \Ext^{\bullet}_{u_{\zeta}(\fg)}(M,M)_{P}\neq 0 \} .
\quad M\in\bK^{c}.
\end{equation} 
Since $M\in \bK^{c}$ it follows that ${V}_{u_{\zeta}(\fg)}(M)$ is a closed $G$-stable subset of $X$ and so $V(M)\in {\mathcal X}_{cl}$. 
 
We now proceed to verify that $V:\bK^{c}\rightarrow {\mathcal X}_{cl}$ as defined above is a quasi support data in the sense of \cref{SS:supportdata}. For properties 
(\ref{SS:supportdata}.\ref{E:supportone})--(\ref{SS:supportdata}.\ref{E:supportfive}) one can apply the arguments given in 
\cite[Section 5]{PW}. 

By \cite{APW2} the Steinberg module and, hence, any projective/injective $U_{\zeta}(\fg)$-module restricts to a projective/injective module for $u_{\zeta}(\fg )$.  Conversely, the authors of \emph{loc.\  cit.} also prove that any 
projective/injective $u_{\zeta}(\fg )$-module is the restriction of a projective/injective $U_{\zeta}(\fg)$-module.  Consequently, to prove 
a $U_{\zeta}(\fg)$-module is projective it suffices to verify it is projective upon restriction to $u_{\zeta}(\fg )$. 
For the sake of self containment, we have included another proof of this fact below. 

\begin{prop} \label{P:projaxiomverify} Let $M$ be a $U_{\zeta}(\fg )$-module. Then $M$ is projective as a $U_{\zeta}(\fg )$-module if and only if it is projective as as a $u_{\zeta}(\fg)$-module.  
\end{prop} 

\begin{proof} If $M$ is projective in $\Mod(U_{\zeta}(\fg))$ then, as discussed above, it is projective when restricted to $u_{\zeta}(\fg)$.

On the other hand, suppose that $M\in \Mod(U_{\zeta}(\fg))$ is projective as a $u_{\zeta}(\fg )$-module.  Let $N$ be an arbitrary module in $\Mod(U_{\zeta}(\fg))$. Consider the Lyndon-Hochschild-Serre spectral sequence 
\begin{equation} 
E_{2}^{i,j}=\Ext^{i}_{G}(\C ,\Ext^{j}_{u_{\zeta}(\fg)}(M,N))\Rightarrow\Ext^{i+j}_{U_{\zeta}(\fg)}(M, N). 
\end{equation} 
Since $M$ is projective over $u_{\zeta}(\fg)$ it follows that $E_{2}^{i,j}=0$ for $j>0$. On the 
other hand, $E_{2}^{i,j}=0$ for all $i>0$ because all rational $G$-modules are completely reducible. Taken together we have that $E_{2}^{i,j}=0$ whenever $i+j>0$.  This immediately implies that
$\Ext^{n}_{U_{\zeta}(\fg)}(M, N)=0$ for all $n>0$. That is, $M$ is projective in $\Mod (U_{\zeta}(\fg ))$.  
\end{proof} 

Now let $M \in \bK^{c}$.  Note that $V(M) =\varnothing$ if and only if $V_{u_{\zeta}(\fg )}(M)=\varnothing$, and that $V_{u_{\zeta}(\fg )}$ satisfies  (\ref{SS:supportdata}.\ref{E:supportseven}).  Combined with the previous result this justifies (\ref{SS:supportdata}.\ref{E:supportseven}) for $V$.  That is, the following result holds true. 

\begin{cor}\label{C:projaxiomverify}
 Let $M\in \bK^{c}$. Then $V(M)=\varnothing$ if and only if $M=0$. 
\end{cor} 

Friedlander \cite{F} defined the notion of mock injectivity for algebraic group schemes over fields of prime characteristic. Hardesty, Nakano and Sobaje \cite{HNS} proved that there are pure mock injective modules for reductive groups (i.e., mock injective modules that are not injective). It is interesting to note that Proposition~\ref{P:projaxiomverify} demonstrates that there are no pure mock injective modules for $U_{\zeta}({\mathfrak g})$. 

\subsection{Realization}
Let $R=\text{H}^{2\bullet}(u_{\zeta}({\mathfrak g}),{\mathbb C})\cong \C[\NN]$. For $M_{1}, M_{2}$ finite-dimensional $u_{\zeta}({\mathfrak g})$-modules, 
let 
$$V^{\max}(M_{1},M_{2})=\text{Maxspec}(R/J(M_{1},M_{2}))$$ 
where $J(M_{1},M_{2})$ is the annihilator ideal of the action of $R$ on $\text{Ext}^{\bullet}_{u_{\zeta}({\mathfrak g})}(M_{1},M_{2})$. The action is obtained by 
taking an extension class in $R=\text{Ext}^{2\bullet}_{u_{\zeta}({\mathfrak g})}({\mathbb C},{\mathbb C})$ and tensoring by $M_{2}$, then splicing this to an extension class in $\text{Ext}^{\bullet}_{u_{\zeta}({\mathfrak g})}(M_{1},M_{2})$. Set $V^{\max}(M_{1})=V^{\max}(M_{1},M_{1})$. 

\begin{prop} \label{P:orbitrealization}
Let ${\OO} \subset \NN$ be a nilpotent $G$-orbit. Then there is a finite-dimensional $U_{\zeta}({\mathfrak g})$-module $M$ such that $V^{\max}(M)=\overline\OO$.
\end{prop}

\begin{proof}
Bezrukavnikov \cite{Bez} proved that there exists a $U_{\zeta}(\fg)$-tilting module $T(w\cdot 0)$ (in the principal block) such that 
$V^{\max}({\mathbb C},T(w\cdot 0))=\overline{\OO}$. We will show that $V^{\max}(T(w\cdot 0))=\overline{\OO}$.

Since 
\begin{equation} \label{E:RelVsAbsSupport}
V^{\max}(T(w\cdot 0)) = V^{\max}(\bigoplus_{\delta\in X_{1}}L(\delta),T(w\cdot 0)) = \bigcup_{\d\in X_{1}} V^{\max}(L(\delta),T(w\cdot 0)), 
\end{equation}
it follows, since $L(0)=\C$, that 
\begin{equation} \label{E:oneinclusion}
\overline{\mathcal O}\subseteq V^{\max}(T(w\cdot 0)).
\end{equation}

To show the reverse inclusion, we will use induced modules and translation functors. Let $C_{\Z}$ be the bottom alcove in $W_{\ell}$. Since $\ell>h$ one has $0\in C_{\Z}$. For $\lambda, \mu\in \overline{C}_{\mathbb Z}$ (closure of the 
bottom alcove), one can define the translation functor $T^{\lambda}_{\mu}$ \cite[II 7.6]{rags}. Since the translation functor is given by the composition of  tensoring by a finite-dimensional $u_{\zeta}({\mathfrak g})$-module followed by projection onto a summand, one has 
$$V^{\max}(T^{\lambda}_{\mu}(M_{1}), T^{\lambda}_{\mu}(M_{2}))\subseteq V^{\max}(M_{1},M_{2}).$$  Also, using the $R$-action described above, one can show that if $N$ is a finite-dimensional $u_{\zeta}({\mathfrak g})$-module then 
\begin{equation} \label{E:tensordecreasesrelativesupport}
V^{\max}(M_{1}\otimes N,M_{2}\otimes N)\subseteq V^{\max}(M_{1},M_{2}).
\end{equation}

We claim that  $V^{\max}(L(\delta),T(w\cdot 0))\subseteq V^{\max}({\mathbb C},T(w\cdot 0))$ for all $\delta\in X_{+}$. 
First observe that 
\begin{equation} \label{E:H0support}
V^{\max}(H^{0}(\delta),T(w\cdot 0)\subseteq V^{\max}({\mathbb C},T(w\cdot 0))
\end{equation} 
for all $\delta\in W_{\ell}\cdot 0 \cap X_{+}$. This is proved by induction on $\delta$. 
For $\delta=0$, this holds because $H^{0}(0)\cong {\mathbb C}$. Suppose  \cref{E:H0support} holds for 
all dominant weights $\delta_{1}<\delta$. According to \cite{rags}, there exists a codimension one facet and a dominant weight $\delta_{1}<\d$ which is a reflection 
of $\delta$ across this facet \cite[II 6.8 Proposition]{rags}. In fact, there is a sequence of affine reflections from $\delta$ to $0\in C_{\mathbb Z}$ such that at each stage the weight 
decreases and remains in the dominant chamber. 

Consider the composition of the translation functors $T^{0}_{\mu}T_{0}^{\mu}$ corresponding to the wall crossing 
functor. Here $\mu$ will be on the wall in the bottom alcove obtained after applying the aforementioned sequence of affine reflections. 
Then $T^{0}_{\mu}T_{0}^{\mu}(H^{0}(\delta_{1}))$ has a good filtration with two composition factors $H^{0}(\delta_{1})$ and 
$H^{0}(\delta)$ \cite[II 7.13 Proposition]{rags}. In fact, there is a short exact sequence, 
$$0\rightarrow H^{0}(\delta_{1}) \rightarrow T^{0}_{\mu}T_{0}^{\mu}(H^{0}(\delta_{1})) \rightarrow H^{0}(\delta) \rightarrow 0.$$ 
Observe using the definition of translation functor, one has 
\begin{eqnarray*} 
V^{\max}(T^{0}_{\mu}T_{0}^{\mu}(H^{0}(\delta_{1})),T(w\cdot 0))&\subseteq & 
V^{\max}(T_{0}^{\mu}(H^{0}(\delta_{1}))\otimes L(\nu),T(w\cdot 0))\\
&\subseteq & V^{\max}(T_{0}^{\mu}(H^{0}(\delta_{1})),T(w\cdot 0)\otimes L(-w_{0}\nu))\\
&\subseteq &V^{\max}(H^{0}(\delta_{1})\otimes L(-w_{0}\nu), T(w\cdot 0)\otimes L(-w_{0}\nu) )\\
&\subseteq &V^{\max}(H^{0}(\delta_{1}),T(w\cdot 0)).
\end{eqnarray*} 
Here $\nu\in W(-\mu) \cap  X_{+}$, and we used \cref{E:tensordecreasesrelativesupport} for the last inclusion. The short exact sequence, the inclusion above, and the induction hypothesis can be now used to show that 
\begin{eqnarray*}
V^{\max}(H^{0}(\delta),T(w\cdot 0)) &\subseteq& V^{\max}(H^{0}(\delta_{1}),T(w\cdot 0)) \cup V^{\max}(T^{0}_{\mu}T_{0}^{\mu}(H^{0}(\delta_{1})),T(w\cdot 0)) \\
&\subseteq& V^{\max}(H^{0}(\delta_{1}),T(w\cdot 0)) \\
&\subseteq & V^{\max}({\mathbb C},T(w\cdot 0)).   
\end{eqnarray*} 

Next, we can again use induction to show that $V^{\max}(L(\delta),T(w\cdot 0))\subseteq V^{\max}({\mathbb C},T(w\cdot 0))$ for all $\delta\in W_{\ell}\cdot 0\cap X_{+}$. 
For $\delta=0$ this holds trivially. So assume it holds for all $\delta_{1}<\delta$ with $\delta_{1}\in W_{\ell}\cdot 0\cap X_{+}$. One has a short exact sequence 
$$0\rightarrow L(\delta)\rightarrow H^{0}(\delta)\rightarrow Q \rightarrow 0$$ 
where all composition factors $L(\delta_{1})$ of $Q$ have the property that $\delta_{1}<\delta$. 
Hence, by applying the induction hypothesis and \cref{E:H0support},  
\begin{eqnarray*}
V^{\max}(L(\delta),T(w\cdot 0)) &\subseteq& V^{\max}(H^{0}(\delta),T(w\cdot 0)) \cup V^{\max}(Q,T(w\cdot 0)) \\
&\subseteq&V^{\max}({\mathbb C},T(w\cdot 0)).   
\end{eqnarray*}

Finally, note that, using the Steinberg tensor product theorem, one can show that any simple module $L(\d)$ in the principal block, where $\d\in X_{1}$, can be realized as a 
$u_{\zeta}({\mathfrak g})$ direct summand of some $L(\d')$ where $\d'\in W_{\ell}\cdot 0\cap X_{+}$. Moreover, by linkage all summands of $T(w\cdot 0)$ upon 
restriction of $u_{\zeta}({\mathfrak g})$ are in the principal block. Consequently, 
$V^{\max}(L(\delta),T(w\cdot 0))\subseteq V^{\max}({\mathbb C},T(w\cdot 0))$ for all $\delta\in X_{+}$. 

Combining this inclusion with \cref{E:RelVsAbsSupport} and \cref{E:oneinclusion}, we have
$V^{\text{max}}(T(w\cdot 0))=\overline{\mathcal O}$. 
\end{proof}

We have $V^{\max}(M)=V(M)\cap \Proj(\text{Maxspec}(\C[\NN]))$. Now the realization property 
(\ref{SS:supportdata}.\ref{E:supporteight}) holds by the argument in \cite[Section 2.4]{BKN}. We also remark that the calculation of $V^{\text{max}}(T(w\cdot 0))$ given in the proof above can be combined with the description of the tensor ideals in the full subcategory of tilting modules (see \cref{S:tilting} and \cite{Ost}) to compute the support variety of any tilting module.  

\subsection{Naturality} Let $R:=\HH^{2\bullet}(u_{\zeta}(\fg),\C )$ and $S=\HH^{2\bullet}(u_{\zeta}(\fb),\C )$ with 
restriction map $\pi:R\rightarrow S$. When $\ell>h$, $R\cong \C[\NN]$, $S\cong S^{\bullet}(\fu^{*})$, $\fu\subseteq \NN$ and $\pi$ is realized by the restriction of functions. 
Therefore, in this case $\pi$ is surjective, and the induced map $\pi^{*}:\Proj(S)\hookrightarrow \Proj( R)$ is a closed injection. For $M\in \mod(u_{\zeta}({\mathfrak b}))$, let 
$${V}_{u_{\zeta}({\mathfrak b})}(M)=\{P\in \Proj (S) \mid \Ext^{\bullet}_{u_{\zeta}({\mathfrak b})}(M,M)_{P}\neq 0 \} .$$ 

Let $J\subseteq \Pi$ and ${\mathfrak p}_{J}={\mathfrak l}_{J}\oplus {\mathfrak u}_{J}$ be the corresponding parabolic subalgebra with ${\mathfrak u}_{J}$ consisting of 
negative root vectors. Let ${\mathfrak a}_{J}={\mathfrak t}\oplus {\mathfrak u}_{J}$. As in the case with ${\mathfrak b}$, for ${\ell}>h$, 
$\HH^{2\bullet}(u_{\zeta}({\mathfrak a}_{J}),{\mathbb C})\cong S^{\bullet}({\mathfrak u}_{J}^{*})$ and $\HH^{2\bullet}(u_{\zeta}({\mathfrak p}_{J}),{\mathbb C})\cong 
{\mathbb C}[P_{J}\times_{B} {\mathfrak u}]$.  Furthermore, one can define support varieties for modules over 
$u_{\zeta}({\mathfrak a}_{J})$ and $u_{\zeta}({\mathfrak p}_{J})$.  

If $M\in \stmod(u_{\zeta}(\fg))$ then under $\pi^{*}$ we will consider $V_{u_{\zeta}(\fb)}(M)$ as a closed set in $\Proj(R)$. Similarly we view
 $V_{u_{\zeta}({\mathfrak p}_{J})}(M)$ as a closed set in $\Proj(R)$.  The restriction maps
\begin{equation*}
\text{res}: \HH^{2\bullet}(u_{\zeta}({\mathfrak p}_{J}),{\mathbb C})\rightarrow \HH^{2\bullet}(u_{\zeta}({\mathfrak a}_{J}),{\mathbb C}) 
\end{equation*} and
\begin{equation*}
\text{res}: \HH^{2\bullet}(u_{\zeta}({\mathfrak p}_{J}),{\mathbb C})\rightarrow \HH^{2\bullet}(u_{\zeta}(\fb),{\mathbb C})
\end{equation*}
 are surjective because they are given by restriction of functions from ${\mathbb C}[P_{J}\times_{B}{\mathfrak u}]
\rightarrow {\mathbb C}[{\mathfrak u}_{J}]$ and ${\mathbb C}[P_{J}\times_{B}{\mathfrak u}]
\rightarrow {\mathbb C}[\fu]$, respectively. Consequently for $N\in  \stmod(u_{\zeta}({\mathfrak p}_{J}))$ restriction again defines inclusions of 
$V_{u_{\zeta}({\mathfrak a}_{J})}(N)$ and $V_{u_{\zeta}({\mathfrak b})}(N)$ in $\Proj({\mathbb C}[P_{J}\times_{B}{\mathfrak u}])$.

One can further refine these inclusions with the 
following result. 

\begin{prop} \label{P:inclusions} Let $M\in \operatorname{stmod}(u_{\zeta}(\fg))$, 
$N\in \operatorname{stmod}(u_{\zeta}({\mathfrak p}_{J}))$. Then one has the following inclusions of support varieties. 
\begin{itemize} 
\item[(a)] $V_{u_{\zeta}(\fb)}(M)\subseteq V_{u_{\zeta}(\fg)}(M)\cap \Proj(S)$;  
\item[(b)] $V_{u_{\zeta}({\mathfrak p}_{J})}(M)\subseteq V_{u_{\zeta}({\mathfrak g})}(M)\cap  \Proj({\mathbb C}[P_{J}\times_{B}{\mathfrak u}])$;  
\item[(c)] $V_{u_{\zeta}({\mathfrak a}_{J})}(N)\subseteq V_{u_{\zeta}({\mathfrak p}_{J})}(N)\cap \Proj(S^{\bullet}({\mathfrak u}_{J}^{*}))$;   
\item[(d)] $V_{u_{\zeta}({\mathfrak b})}(N)\subseteq V_{u_{\zeta}({\mathfrak p}_{J})}(N)\cap \Proj(S)$.   
\end{itemize} 

\end{prop} 

\begin{proof} (a) First consider the commutative diagram:
$$
\CD
R = \HH^{2\bullet}(u_{\zeta}(\fg),\C )  @>>>
\Ext^{\bullet}_{u_{\zeta}(\fg)}(M,M)\\
@VV\pi  V  @VV \sigma V \\
S = \HH^{2\bullet}(u_{\zeta}(\fb),\C )@>>>
\Ext^{\bullet}_{u_{\zeta}(\fb)}(M,M),
\endCD
$$
where each horizontal map is given by the action of the ring on the identity morphism, and each vertical map is restriction.
Let $J_\fg(M)$ be the annihilator of $R$ on $\Ext^{\bullet}_{u_{\zeta}(\fg)}(M,M)$, and 
$J_\fb(M)$ be the annihilator of $S$ on $\Ext^{\bullet}_{u_{\zeta}(\fb)}(M,M)$. The commutativity of the 
diagram implies that $\pi(J_\fg(M))\subseteq J_\fb(M)$. Therefore, 
$J_\fg(M)\subseteq \pi^{-1}(J_\fb(M))$. Now if $J_\fb(M)\subseteq P$ where $P\in \Proj(S)$ then 
$J_\fg(M)\subseteq \pi^{-1}(P) = \pi^{*}(P)$. This shows that $V_{u_{\zeta}(\fb)}(M)\subseteq V_{u_{\zeta}(\fg)}(M)\cap \Proj(S)$.  
The cases (b)--(d) are proved in a similar manner. 
\end{proof}

The following proposition is our first result which shows  support varieties behave well with respect to restriction to certain quantum 
subalgebras. 

\begin{prop} \label{P:naturalityp} Let ${\mathfrak p}_{J}$ be a parabolic subalgebra in ${\mathfrak g}$ containing ${\mathfrak b}$, and let $N\in \operatorname{stmod}(u_{\zeta}({\mathfrak p}_{J}))$. Then 
$$V_{u_{\zeta}({\mathfrak a}_{J})}(N)=V_{u_{\zeta}({\mathfrak p}_{J})}(N)\cap  \Proj(S^{\bullet}({\mathfrak u}_{J}^{*})).$$ 
\end{prop} 

\begin{proof} From \cref{P:inclusions}(c), $V_{u_{\zeta}({\mathfrak a}_{J})}(N)\subseteq V_{u_{\zeta}(\fp_{J})}(N)\cap  \Proj(S^{\bullet}({\mathfrak u}_{J}^{*}))$.

We next prove the other inclusion.  For $N\in \operatorname{stmod}(u_{\zeta}({\mathfrak p}_{J}))$ one has a  Lyndon-Hochschild-Serre spectral sequence: 
\begin{equation*} 
E^{i,j}_{2}(N)=\HH^{i}(u_{\zeta}({\mathfrak l}_{J}),\HH^{j}(u_{\zeta}({\mathfrak u}_{J}),N^{*}\otimes N))\Rightarrow  \HH^{i}(u_{\zeta}({\mathfrak p}_{J}),N^{*}\otimes N)
\end{equation*} 
with an action of $E^{\bullet,\bullet}_{2}({\mathbb C})$ on $E^{\bullet,\bullet}_{2}(N)$. Furthermore, the (vertical) edge homomorphism for $N=\C$ can be realized as the restriction map: 
$$\HH^{\bullet}(u_{\zeta}({\mathfrak p}_{J}),{\mathbb C})\rightarrow E_{2}^{0,\bullet}({\mathbb C})=\HH^{\bullet}(u_{\zeta}({\mathfrak u}_{J}),{\mathbb C})^{u_{\zeta}({\mathfrak l}_{J})}.$$ 
From the discussion in \cite[Section 5.6]{BNPP} this edge homomorphism factors through the restriction map  $\res : \HH^{2\bullet}(u_{\zeta}({\mathfrak p}_{J}),{\mathbb C}) \to \HH^{2\bullet}(u_{\zeta}({\mathfrak a}_{J}),{\mathbb C})$ and so $\HH^{2\bullet}(u_{\zeta}({\mathfrak a}_{J}),{\mathbb C})$ acts on the pages of the spectral sequence $E^{i,\bullet}_{r}(N)$. 

Next observe  
\begin{equation}\label{E:sumsum}
\text{Ext}^{\bullet}_{u_{\zeta}({\mathfrak a}_{J})}(\oplus \lambda,N^{*}\otimes N)\cong \text{Ext}^{\bullet}_{u_{\zeta}({\mathfrak u}_{J})}({\mathbb C},N^{*}\otimes N), 
\end{equation}
where $\oplus \lambda$ is the direct sum of the simple $u_{\zeta}({\mathfrak a}_{J})$-modules. Let $P\in  \Proj(S^{\bullet}({\mathfrak u}_{J}^{*}))$ with $P\notin V_{u_{\zeta}({\mathfrak a}_{J})}(N)$. 
Then from \cref{E:sumsum} we have
$$E^{i,\bullet}_{2}(N)_{P}=\HH^{i}(u_{\zeta}({\mathfrak l}_{J}),\HH^{\bullet}(u_{\zeta}({\mathfrak u}_{J}),N^{*}\otimes N)_{P})=0$$ 
for all $i\geq 0$.
Consequently, $\HH^{\bullet}(u_{\zeta}({\mathfrak p}_{J}),N^{*}\otimes N)_{P}=0$. This proves $V_{u_{\zeta}({\mathfrak a}_{J})}(N)\supseteq V_{u_{\zeta}({\mathfrak p})}(N)\cap  \Proj(S^{\bullet}({\mathfrak u}_{J}^{*}))$.
\end{proof} 

Next we claim $P\in \pi^{-1}(\text{Spec}(S))\subseteq \text{Spec}(R)$ if and only if $Z(P)\subseteq {\mathfrak u}$. Here $Z(P)$ 
is the zero locus of $P$. To see this, suppose that $P=\pi^{-1}(P^{\prime})\supseteq \pi^{-1}(\{0\})$. Then $Z(P)\subseteq Z(\pi^{-1}(\{0\}))={\mathfrak u}$. Conversely, let $Z(P)\subseteq {\mathfrak u}$. 
Since $P$ is prime $Z(P)$ is irreducible and there exists a prime ideal $P^{\prime}\in \text{Spec}(S)$ with 
$Z(P^{\prime})=Z(\pi^{-1}(P^{\prime}))=Z(P)$. Now since $R$ is reduced, prime ideals are equal to their radicals and 
so $P=\pi^{-1}(P^{\prime})$. The same proof also shows that $P\in \pi^{-1}(\text{Spec}(S^{\bullet}({\mathfrak u}_{J}^{*}))\subseteq \text{Spec}(S)$ if and only if 
$Z(P)\subseteq {\mathfrak u}_{J}$.

With our identifications, we have the following result involving the naturality of support varieties. 

\begin{theorem} \label{T:naturality} Let $M\in \operatorname{stmod}(U_{\zeta}(\fg))$. Then 
$$V_{u_{\zeta}(\fb)}(M)=V_{u_{\zeta}(\fg)}(M)\cap  \Proj(S).$$ 
\end{theorem} 

\begin{proof} According to \cref{P:inclusions}(a), $V_{u_{\zeta}(\fb)}(M)\subseteq V_{u_{\zeta}(\fg)}(M)\cap \Proj(S)$.  
In order to prove the other inclusion, it suffices to verify the following statement: 
\vskip .15cm 
\noindent 
{\em (*) If $P\in V_{u_{\zeta}(\fb)}(M)$ (regarded in $V_{u_{\zeta}(\fg)}(M)$) and 
$wP\in \Proj(S)$ for $w\in W$ then $wP\in V_{u_{\zeta}(\fb)}(M)$.} 
\vskip .15cm
 
Suppose the statement holds. Let $P\in V_{u_{\zeta}(\fg)}(M)\cap \Proj(S)$. 
Since $G\cdot V_{u_{\zeta}(\fb)}(M)=V_{u_{\zeta}(\fg)}(M)$ \cite[Theorem 6.1]{Dr} it follows using the Bruhat decomposition that $P=(b_{1}wb_{2})P^{\prime}$ where 
$P^{\prime}\in V_{u_{\zeta}(\fb)}(M)$ and $b_{1}, b_{2} \in B$. Furthermore, $b_{1}^{-1}P=wb_{2}P^{\prime}\in \Proj(S)$. Set $P^{\prime \prime}=b_{2}P^{\prime}$. Since $\Proj (S)$ is $B$-stable it follows that $wP^{\prime \prime} = b_{1}^{-1}P \in \Proj(S)$. Since $V_{u_{\zeta}(\fb)}(M)$ is $B$-stable one has $P^{\prime \prime} \in V_{u_{\zeta}(\fb)}(M)$.  Now 
by (*), $wP^{\prime \prime} \in V_{u_{\zeta}(\fb)}(M)$, thus $wb_{2}P^{\prime}\in  V_{u_{\zeta}(\fb)}(M)$ and 
by $B$-stability, $b_{1}wb_{2}P^{\prime}\in  V_{u_{\zeta}(\fb)}(M)$. Consequently, $P\in V_{u_{\zeta}(\fb)}(M)$.

We now prove (*) by induction on the length of $w\in W$. For $l(w)=0$, the statement is trivially true. 
If $l(w)=1$ then $w=s_{\alpha}$ for some $\alpha\in \Pi$. Suppose that $P\in V_{u_{\zeta}(\fb)}(M)$ and 
$s_{\alpha}P\in \Proj(S)$. Then $Z(s_{\alpha}P)\subseteq {\mathfrak u}$. This implies that $Z(P)\subseteq {\mathfrak u}_{J}$ where 
$J=\{\alpha\}$ and $P\in \Proj(S^{\bullet}({\mathfrak u}_{J}^{*}))$. It follows that $s_{\alpha}P\in  \Proj(S^{\bullet}({\mathfrak u}_{J}^{*}))$ 
and $s_{\alpha}P\in V_{u_{\zeta}({\mathfrak p}_{J})}(M) \cap \Proj(S^{\bullet}({\mathfrak u}_{J}^{*}))$. Therefore by  \cref{P:naturalityp} it follows  
$s_{\alpha}P\in V_{u_{\zeta}({\mathfrak a}_{J})}(M)$. On the other hand, arguing as in the proof of \cref{P:inclusions} shows $V_{u_{\zeta}({\mathfrak a}_{J})}(M)\subseteq  V_{u_{\zeta}({\mathfrak b})}(M)$ and the desired result follows. 

Now assume the statement holds for all $w^{\prime}\in W$ with $l(w^{\prime})<l(w)=t$. Let $w=s_{\alpha_{1}}s_{\alpha_{2}}\dots s_{\alpha_{t}}$ 
be a reduced expression. Set $w^{\prime}=s_{\alpha_{1}}s_{\alpha_{2}}\dots s_{\alpha_{t-1}}$. Assume $P\in V_{u_{\zeta}(\fb)}(M)$ and $wP\in \Proj(S)$. Then $Z(P) \subseteq \fu$ and $wZ(P) = Z(wP)\subseteq {\mathfrak u}$.  According to \cite[Corollary 10.2C]{Humphreys}, $w(\alpha_{t})=s_{\alpha_{1}}s_{\alpha_{2}}\dots s_{\alpha_{t}}(\alpha_{t})<0$, thus $Z(P) \subseteq {\mathfrak u}_{\{\alpha_{t}\}}$. Consequently, $ Z(s_{\alpha_{t}}P)=s_{\alpha_{t}}Z(P)  \subseteq {\mathfrak u}_{\{\alpha_{t}\}} \subseteq \fu$ and so $s_{\alpha_{t}}P\in  \Proj(S)$. From the induction hypothesis, $s_{\alpha_{t}}P\in V_{u_{\zeta}(\fb)}(M)$. Now set $P^{\prime}=s_{\alpha_{t}}P$. 
Then $P^{\prime}\in V_{u_{\zeta}(\fb)}(M)$ and $w^{\prime}P^{\prime} =wP \in \Proj(S)$. Again by the induction hypothesis, $w^{\prime}P^{\prime}\in V_{u_{\zeta}(\fb)}(M)$.  That is,  $wP\in V_{u_{\zeta}(\fb)}(M)$ as desired. 
\end{proof} 

\subsection{Verifying Assumption \texorpdfstring{~\ref{A:projectivity}}{2.5.1}} The following theorem demonstrates that Assumption~\ref{A:projectivity} holds for 
$\bK$.  In the proof we freely make use of the fact that if $M$ is an object of $\stmod (u_{\zeta}(\fg ))$ (resp.\ $\stmod (u_{\zeta}(\fb ))$), then $W_{u_{\zeta}(\fg)}(M)=V_{u_{\zeta}(\fg )}(M)=\ZZ_{u_{\zeta}(\fg )}(M)$ (resp.\ $W_{u_{\zeta}(\fb)}(M)=V_{u_{\zeta}(\fb )}(M)=\ZZ_{u_{\zeta}(\fb )}(M)$) \cite[Theorem 5.5]{BIK}.  In particular, $W_{u_{\zeta}(\fg)}(-)$ and $W_{u_{\zeta}(\fb)}(-)$ satisfy properties  (\ref{SS:supportdata}.\ref{E:supportone})--(\ref{SS:supportdata}.\ref{E:supportfive}) whenever all objects involved are compact. When $M$ is not compact the containment $W_{u_{\zeta}(\fb)}(M) \subseteq \ZZ_{u_{\zeta}(\fb )}(M)$ still holds by \cite[Remark 5.4]{BIK}.  Furthermore, $ \ZZ_{u_{\zeta}(\fb )}(-)$ satisfies properties (\ref{SS:supportdata}.\ref{E:supportone})--(\ref{SS:supportdata}.\ref{E:supportfour}).

\begin{theorem} \label{P:projectivityverify} Let $\bK=\operatorname{Stmod}(U_{\zeta}(\fg))$, let $M\in \bK^{c}$ with $M\neq 0$, $\bI'=\Tensor(M)$, and $N\in \bI_{V(M)}$ (recall \cref{E:IWdefinition}). 
If  $M\otimes L_{\bI'} (N)=0$ then  $L_{\bI'} (N)=0$. 
\end{theorem}

\begin{proof}Set $A=u_{\zeta}(\fb)$ and let $\bar A = \gr A$ be the associated graded algebra as described in Section~\ref{SS:grb}. Combining Theorem~\ref{T:proj1} and the discussion at the end of \cref{SS:modulecategories}, it suffices to show that 
$L_{\bI'} (N)=0$ in $\Stmod(A)$. 

From our assumption, $M\otimes L_{\bI'} (N)=0$ and so $ \gr \left( M\otimes L_{\bI'} (N) \right)=0$ in $\Stmod (\bar{A})$. Therefore,
$$\gr M \otimes \gr  L_{\bI'}(N) =0$$ 
in $\Stmod(\Delta)$. 
According to Theorem~\ref{T:grWintersection}, 
\begin{equation*}
W_{\bar A}(\gr M)\cap W_{\bar A}(\gr L_{\bI'}(N))=\varnothing.
\end{equation*}
Applying \cref{T:Winclusion} yields, 
\begin{equation}\label{eq:tensorprod}
W_{A}(M)\cap W_{A}(L_{\bI'}(N))=\varnothing.
\end{equation}

Now suppose that $P\notin W_{A}(M)$. Set $\bI=\bI_{V(M)}$. We first claim that $P\notin W_{A}(Q)=V_{A}(Q)=\ZZ_{A}(Q)$ 
for any $Q\in \bI$.  For suppose that $P\in W_{A}(Q)=V_{A}(Q)$ for some $Q\in \bI$. Then 
$$P\in V_{A}(Q)\subseteq V_{u_{\zeta}(\fg)}(Q)\subseteq V_{u_{\zeta}(\fg)}(M) =G\cdot V_{A}(M).$$
 The first containment is \cref{T:naturality}, the second holds because $Q \in \bI$ and by definition of $ \bI_{V(M)}$,
 %$V_{u_{\zeta}(\fg)}(-)$ satisfies  (\ref{SS:supportdata}.\ref{E:supportone})--(\ref{SS:supportdata}.\ref{E:supportfive}), 
 and the equality follows from \cite[Theorem 6.1]{Dr}.  Thus
 there exists $P^{\prime}\in V_{A}(M)$ with $P=g\cdot P^{\prime}$ for some $g\in G$. Using the Bruhat 
decomposition: $G=\bigcup_{w\in W} BwB$, it follows that $P=bwb^{\prime}\cdot P^{\prime}$ for some $b, b^{\prime}\in B$ and $w\in W$. 

Since $V_{A}(M)$ and $V_{A}(Q)$ are $B$-invariant, we have $\widetilde P := b'\cdot P' \in V_{A}(M)$ with 
$w\widetilde P = b^{-1} \cdot P  \in V_{A}(Q) \subseteq \Proj(S)$. On the other hand $w\widetilde P \in G \cdot V_{A}(M)=V_{u_{\zeta}(\fg )}(M)$.  Taken together with Theorem~\ref{T:naturality}, 
$$w\widetilde P\in V_{u_{\zeta}(\fg)}(M)\cap  \Proj(S)=V_{A}(M).$$ That is, $b^{-1}\cdot P \in V_{A}(M)$ and, by the $B$-stability of $V_{A}(M)$, $P \in V_{A}(M)$.  Since by assumption $P \not\in W_{A}(M)$, it must be that $P\notin W_{A}(Q)$ for any $Q\in \bI$. 

As in the proof of  \cref{T:Hopkins}, we have $\bI' \subseteq \bI$, and $Q\in \bI$ implies $L_{\bI'}(Q)\in \operatorname{Loc}(\bI)$. Since $W_{A}(-)$ satisfies (\ref{SS:supportdata}.\ref{E:supportone})--(\ref{SS:supportdata}.\ref{E:supportfour}), the preceding result along with \cite[Lemma 2.4.1]{BKN} implies that $P\notin W_{A}(L_{\bI'}(Q))$ for all $Q\in \bI$; and in particular, for $Q=N$.   That is, we have shown that $W_{A}(L_{\bI'}(N)) \subseteq W_{A}(M)$.  

Combining this inclusion with \cref{eq:tensorprod} yields $W_{A}(L_{\bI'}(N))=\varnothing$ and so $L_{\bI'} (N)=0$  by \cite[Theorem 5.2]{BIK}.
\end{proof}

\subsection{Classification of Thick Tensor Ideals}   With  our verifications in the prior sections we can present a complete classification of the (thick) tensor ideals in the category $\stmod (U_{\zeta}(\fg))$ using Theorem~\ref{I:bijectiongeneral}. 

\begin{theorem} \label{T:classification}   Let $G$ be a complex simple algebraic group over $\C $ with $\fg=\operatorname{Lie }G$. 
Assume that $\zeta$ is a primitive $\ell$th root of unity where $\ell>h$. Let $\bK=\operatorname{Stmod}(U_{\zeta}(\fg))$, 
$\bK^{c}=\operatorname{stmod}(U_{\zeta}(\fg))$ and $X=G\text{-}\Proj(\C [\NN])$. 
Adopt the notation of \cref{SS:zariski}.  Let $V:\bK^{c} \rightarrow \XX$ be the quasi support data on $\bK^{c}$ defined in (\ref{e:supportdatadef}). 
There is a pair of mutually inverse maps
$$
\{\text{thick tensor ideals of $\bK^{c}$}\} \begin{array}{c} \gf{\Gamma}{\longrightarrow} \\ \gf{\longleftarrow}{\Theta} \end{array}  \XX_{sp},
$$
given by 
\begin{align*}
\Gamma(\bI)=\bigcup_{M\in \bI} V(M), \qquad \Theta(W)= \bI_{W},
\end{align*}
where $\bI_{W}=\{M\in \bK^{c} \mid V(M)\subseteq W \}$. 
\end{theorem}

\subsection{Computation of \texorpdfstring{$\operatorname{Spc}(\stmod(U_{\zeta}(\fg)))$}{Spc(stmod(U\_zeta(g)))}} 
We can apply Theorem~\ref{K:bijectiongeneral} to identify the 
Balmer spectrum for the stable module category for the quantum group $U_{\zeta}(\fg)$. 

\begin{theorem} \label{T:spcquantum} Let $G$ be a complex simple algebraic group over $\C $ with $\fg=\operatorname{Lie }G$. 
Assume that $\zeta$ is a primitive $\ell$th root of unity where $\ell>h$. Then there is a homeomorphism
$$
\operatorname{Spc}(\operatorname{stmod}(U_{\zeta}(\fg))) \cong G\text{-}\Proj(\C [\NN]).$$  
\end{theorem}

\subsection{The category \texorpdfstring{$\mod(U_{\zeta}(\fg))$}{mod(U\_zeta(g))}} Let $R=\C[\NN]$. Following the work in \cite{Lorenz2009} and \cite{Lorenz2013}, the rational ideals of $R$ are precisely the maximal ideals of $R$. Furthermore, the points in $G\text{-}\operatorname{MaxSpec}(\C [\NN])$ are in bijective correspondence with $G$-orbits in $\operatorname{MaxSpec}(\C [\NN]))$ (i.e., the nilpotent orbits in $\NN$). It follows that $G\text{-}\operatorname{MaxSpec}(\C [\NN])$ is finite and 
$$G\text{-}\operatorname{MaxSpec}(\C [\NN])=G\text{-}\operatorname{Spec}(\C [\NN])$$ by \cite[Proposition 1]{Lorenz2013}. Furthermore, $\rho$ defines an inclusion preserving bijection between the $G$-stable closed sets of $\NN$ and the closed sets of $G\text{-}\operatorname{Spec}(\C [\NN])$ (see \cite[Section 2.3]{BKN}).  This allows us to lift results from $\bK^{c}$ to $\mod (U_{\zeta}(\fg ))$. 

\begin{prop}\label{P:TensorProductPropertyforUzeta}  Let $M$ and $N$ be modules in  $\mod (U_{\zeta}(\fg ))$.  Then 
\[
V_{u_{\zeta}(\fg )}\left( M \otimes N \right) = V_{u_{\zeta}(\fg )}\left( M \right) \cap V_{u_{\zeta}(\fg )}\left( N \right).
\]

\end{prop}

\begin{proof} From the fact that $V$ satisfies the tensor product property (\cref{P:TensorProductProperty}) and  $V_{u_{\zeta}(\fg)}$ satisfies (\ref{SS:supportdata}.\ref{E:supportfour}), there is the following chain of containments:
\begin{align*}
\rho\left(V_{u_{\zeta}(\fg )}\left( M \right) \cap V_{u_{\zeta}(\fg )}\left( N \right) \right) &\subseteq \rho\left(V_{u_{\zeta}(\fg )}\left( M \right) \right) \cap \rho\left(  V_{u_{\zeta}(\fg )}\left( N \right)\right)\\
   &= V(M) \cap V(N) \\
   & = V(M \otimes N) \\
    & = \rho \left(V_{u_{\zeta}(\fg )}\left( M \otimes N \right) \right) \\
     &\subseteq \rho \left( V_{u_{\zeta}(\fg )}\left( M \right) \cap V_{u_{\zeta}(\fg )}\left( N \right)\right)
\end{align*}

In particular, $\rho(V_{u_{\zeta}(\fg )}( M ) \cap V_{u_{\zeta}(\fg )}( N ) ) = \rho (V_{u_{\zeta}(\fg )}( M \otimes N ) )$.  However, both $V_{u_{\zeta}(\fg )}( M ) \linebreak \cap V_{u_{\zeta}(\fg )}( N )$ and $V_{u_{\zeta}(\fg )}( M \otimes N )$ are $G$-stable closed sets in $\NN$ and $\rho$ is a bijection on such sets and so they must be equal.
\end{proof}

We call a full subcategory $I$ of $\mod(U_{\zeta}(\fg))$ a thick tensor ideal if 1) given any $M \in I$ and $N \in \mod(U_{\zeta}(\fg))$, $M \otimes N \in I$, and 2) if $M \oplus N \in I$, then $M$ and $N$ are in $I$, and 3) if $0\to A \to B \to C \to 0$ is an exact sequence with two of $A,B,C$ in $I$, then the third is in $I$.  We call a proper tensor ideal prime if $M \otimes N \in I$ implies $M \in I$ or $N \in I$ for all $M,N \in \mod (U_{\zeta}(\fg ))$.  From Theorem~\ref{T:classification}, we can also classify the thick tensor ideals and thick prime ideals in $\operatorname{mod}(U_{\zeta}(\fg))$. 

The thick tensor ideals are ordered by inclusion as are the $G$-stable closed subsets of $\NN$. Recall that there is also a partial order on the nilpotent orbits given by declaring $G.x \preceq G.y$ (for $x, y\in \NN$) if and only if $\overline{G.x}\subseteq \overline{G.y}$.

\begin{cor} \label{C:thickprime} Let $G$ be a complex simple algebraic group over $\C $ with $\fg=\operatorname{Lie }G$. 
Assume that $\zeta$ is a primitive $\ell$th root of unity where $\ell>h$. Then there exist order preserving bijective correspondences
$$
\{\text{thick tensor ideals of $\operatorname{mod}(U_{\zeta}(\fg))$}\}   
\begin{array}{c} \gf{\phantom{\Gamma}}{\longrightarrow} \\ \gf{\longleftarrow}{\phantom{\Theta}} \end{array}
\{\text{$G$-stable closed subsets of $\NN$}\},
$$
and
$$
\{\text{nonzero thick prime tensor ideals of $\operatorname{mod}(U_{\zeta}(\fg))$}\}   
\begin{array}{c} \gf{\phantom{\Gamma}}{\longrightarrow} \\ \gf{\longleftarrow}{\phantom{\Theta}} \end{array}
\{\text{nilpotent $G$-orbits in $\NN$}\}.
$$
\end{cor} 

\begin{proof} 
First, note that the zero ideal is a prime thick tensor ideal of $\mod (U_{\zeta}(\fg))$ which is contained in every other thick tensor ideal.  

Second, the full subcategory $\mathcal{P}$ of all projective modules in $\mod (U_{\zeta}(\fg ))$ is a prime thick tensor ideal. That $\mathcal{P}$ is prime follows from \cref{P:TensorProductPropertyforUzeta}, the fact that $\NN$ has a unique minimal nonzero orbit, and that $V_{u_{\zeta}(\fg )}(M) = \varnothing$ if and only if $M$ is projective. 

Furthermore, note that $\mathcal{P}$ is contained in every nonzero thick tensor ideal $I$ of  $\mod (U_{\zeta}(\fg ))$.  Namely, given any nonzero $M \in I$ and nonzero $P \in \mathcal{P}$, $M \otimes P \in I \cap \mathcal{P}$ and thus $I$ contains nonzero projective modules.  Thus it suffices to show that any nonzero $Q$ in $\mathcal{P}$ generates $\mathcal{P}$ as a thick tensor ideal.  If we write $\text{ev} : Q^{*} \otimes Q \to \C$ for the evaluation homomorphism, then there is a surjective homomorphism $\text{ev} \otimes 1: Q^{*}\otimes Q \otimes P \to P$ for any $P \in \mathcal{P}$.  Since $P$ is projective this map splits and so $P$ is isomorphic to a direct summand of $Q^{*}\otimes Q \otimes P$ and thus every projective module is contained in the thick tensor ideal generated by $Q$ as needed.

We now observe that there is a one-to-one correspondence between the nonzero thick tensor ideals of  $\mod (U_{\zeta}(\fg ))$ and the thick tensor ideals of $\bK^{c}= \Stab \left(\mod (U_{\zeta}(\fg )) \right)$.  Namely, given a nonzero thick tensor ideal $I$ of $\mod (U_{\zeta}(\fg ))$, let $I'$ denote the full subcategory of $\bK^{c}$ whose objects are those of $I$.  Using that $\mathcal{P} \subseteq I$ and the defining properties of a thick tensor ideal in $\mod (U_{\zeta}(\fg ))$ it follows that $I'$ is a thick tensor ideal of $\bK^{c}$.  Conversely, given a thick tensor ideal, $J'$, of $\bK^{c}$, let $J$ be the full subcategory of $\mod (U_{\zeta}(\fg ))$ whose objects are the objects of $J'$. This forms a nonzero thick tensor ideal in $\mod (U_{\zeta}(\fg))$.  These provide mutually inverse, order preserving bijections between the nonzero thick tensor ideals of $\mod (U_{\zeta}(\fg ))$ and the thick tensor ideals of of $\bK^{c}$.   Combining this with the discussion at the beginning of this section and the bijection between the closed subsets of $G$-$\Proj (\C [\NN ])$ and the nonzero $G$-stable closed subsets of $\NN$ given by $\rho$ yields the first correspondence, after modifying the bijection slightly by sending $\mathcal P$ to $\{0\}\subset \NN$ (instead of to $\varnothing$), and extending it by sending the zero ideal to $\varnothing \subset \NN$.

 Moreover, the above bijection between the nonzero thick tensor ideals of $\mod (U_{\zeta}(\fg ))$ and the thick tensor ideals of $\bK^{c}$ restricts to a bijection between the set of prime ideals of $\bK^{c}$ and the set of nonzero prime ideals of $\mod (U_{\zeta}(\fg ))$.  That is, there is a bijection between the nonzero prime ideals of $\mod (U_{\zeta}(\fg ))$ and the irreducible $G$-stable closed subsets of $\NN$ and, hence, with the nilpotent $G$-orbits in $\NN$. 
\end{proof} 

\section{Connections with Tilting Modules} \label{S:tilting}

\subsection{} Let $\bT$ be the category of finite-dimensional tilting modules for the braided monoidal tensor category $\mod(U_{\zeta}(\fg))$.  As it is closed under the tensor product and contains the unit object, $\bT$ is itself a braided monoidal tensor category. Write $T(\lambda)$ for the indecomposable tilting module of highest weight $\lambda\in X_{+}$, and  $\Tensor_{\bT}(T)$ for the thick tensor ideal in $\bT$ generated by a tilting module $T$. Ostrik provided a classification of the thick tensor ideals in $\bT$. Combined with Bezrukavnikov's work on the support varieties of tilting modules we can state the following result. 

\begin{theorem} \cite{Ost, Bez} \label{T:Ostrikclass} Let $\lambda,\mu\in X_{+}$. 
\begin{itemize} 
\item[(a)] There exist bijective correspondences
$$
\{\operatorname{Tensor}_{\bT}(T(\lambda))\mid \lambda\in X_{+}\}  \begin{array}{c} \gf{\phantom{\Gamma}}{\longrightarrow} \\ \gf{\longleftarrow}{\phantom{\Theta}} \end{array} 
\{\text{nilpotent orbits of $G$}\} \begin{array}{c} \gf{\phantom{\Gamma}}{\longrightarrow} \\ \gf{\longleftarrow}{\phantom{\Theta}} \end{array} \{\text{two-sided cells of $W_{\ell}$} \}.  
$$
\item[(b)] $T(\mu)\in \operatorname{Tensor}_{\bT}(T(\lambda))$ if and only if $V(T(\mu))\subseteq V(T(\lambda))$. 
\end{itemize} 
\end{theorem}

\subsection{} It is interesting to compare Ostrik's classification with our results presented in Corollary~\ref{C:thickprime}. Indeed, the 
following results indicates that $\bR:=\mod(U_{\zeta}(\fg))$ is an ``integral extension" of $\bT$. 

\begin{theorem} Let $\lambda,\mu\in X_{+}$. Then 
\begin{itemize} 
\item[(a)] $\operatorname{Tensor}_{\bT}(T(\lambda))=\operatorname{Tensor}_{\bR}(T(\lambda))\cap \bT$.
\item[(b)] If $I$ is a thick prime tensor ideal in $\bR$, then $I=\operatorname{Tensor}_{\bR}(T(\lambda))$ for some $\lambda\in X_{+}$. 
\item[(c)] $\operatorname{Tensor}_{\bR}(T(\lambda))=\operatorname{Tensor}_{\bR}(T(\mu))$ if and only if 
$V(T(\lambda))=V(T(\mu))$. 
\end{itemize} 
\end{theorem} 

\begin{proof}  For part (a), 
by using the definitions, it is clear that 
$$\operatorname{Tensor}_{\bT}(T(\lambda))\subseteq \operatorname{Tensor}_{\bR}(T(\lambda))\cap \bT.$$ 
Now suppose that $T(\mu)\in \operatorname{Tensor}_{\bR}(T(\lambda))\cap \bT$. Then 
$V(T(\mu))\subseteq V(T(\lambda))$. Now one can apply Theorem~\ref{T:Ostrikclass}(b) to see that 
$T(\mu)\in \operatorname{Tensor}_{\bT}(T(\lambda))$. 

Parts (b) and (c) are consequences of part (a), Theorem~\ref{T:classification}, Corollary~\ref{C:thickprime}, and \cref{T:Ostrikclass}.
\end{proof}

\section{Appendix: Graded Algebras and Modules} \label{S:graded}

\subsection{} \label{S:gradings} Let $A$ be a finite-dimensional algebra with an increasing filtration: 
$$A_{0}\subseteq A_{1} \subseteq \dots   \subseteq A_{N}=A$$ 
which is multiplicative (i.e., $A_{i}A_{j}\subseteq A_{i+j}$). Moreover, let 
$$\gr A=A_{0}\oplus A_{1}/A_{0} \oplus \dots \oplus A_{N}/A_{N-1}$$ 
be the associated graded algebra. 

If $M$ is an $A$-module then one can construct $\gr M$ which is a 
$\gr A$-module as follows. Fix a set of generators $\{m_{i}\mid i\in I \}$ of $M$ and set $M_{j}=\sum_{i\in I}A_{j}m_{i}$. One 
obtains an increasing filtration 
$$M_{0}\subseteq M_{1}\subseteq \dots \subseteq M_{N}=M.$$ 
with the property that $A_{i}\cdot M_{j}\subseteq M_{i+j}$. Set  
$$\gr M=M_{0}\oplus M_{1}/M_{0} \oplus \dots \oplus M_{N}/M_{N-1}$$ 
which is a $\gr A$-module. 

\subsection{}\label{SS:associatedgraded} Now let $A$ and $A^{\prime}$ be finite-dimensional algebras with multiplicative filtrations given by 
$A_{0}\subseteq A_{1}\subseteq \dots \subseteq A_{N}=A$ and $A^{\prime}_{0}\subseteq A^{\prime}_{1}\subseteq \dots \subseteq  A^{\prime}_{N^{\prime}}=A^{\prime}$. 
Let $B=A\otimes A^{\prime}$ and $B_{t}=\sum_{i+j=t} A_{i}\otimes A^{\prime}_{j}$. Then we have a multiplicative filtration $B_{0}\subseteq B_{1}\subseteq \dots   \subseteq B_{N+N^{\prime}}=B$ 
for the algebra $B$. One can consider the associated graded algebra $\gr B$. 

The multiplication on $\gr B$ (i.e., $B_{s}/B_{s-1}\times B_{t}/B_{t-1} \rightarrow B_{s+t}/B_{s+t-1}$) is given by the formula 
$$[(x \otimes x^{\prime} )+B_{s-1}][(y\otimes y^{\prime})+B_{t-1}]=(xy\otimes x^{\prime}y^{\prime})+B_{s+t-1}.$$ 
Suppose we have $x\in A_{i}$, $y\in A_{j}$ with $i+j=s$ and $x^{\prime}\in A_{i^{\prime}}$, $y\in A_{j^{\prime}}$ with $i^{\prime}+j^{\prime}=t$. 
Then replacing $x$ by another element in $A_{i}$ modulo $A_{i-1}$ does not change the product. The analogous statement is true for $y$, $x^{\prime}$ and $y^{\prime}$. 
Therefore, one has a canonical isomorphism 
$$\gr (A\otimes A^{\prime})\cong \gr A \otimes \gr A^{\prime}.$$ 
Similarly, if $M$ (resp.\ $M^{\prime}$) is an $A$-module (resp.\ $A^{\prime}$-module), using the associated filtrations one obtains 
\begin{equation}\label{E:tensorprodgrad}
\gr (M\boxtimes M^{\prime})\cong \gr M\boxtimes \gr M^{\prime}
\end{equation} 
as $\gr A \otimes \gr A^{\prime}$-modules, where $\boxtimes$ denotes the outer tensor product. 

Let $C$ be a subalgebra of $B$. The multiplicative filtration on $B$ induces a multiplicative filtration on $C$ by letting $C_{i}=C\cap B_{i}$ for all $i$. One 
can consider the associated graded algebras, and there is an injection of subalgebras: $\gr C\leq \gr B$. 

\subsection{Spectral Sequences} Let $A$ be a finite-dimensional augmented algebra with augmentation ideal $A_{+}$. 
Following the construction in \cite[Lemma 5.6.1]{BNPP} and \cite{Bajer}, let $A_{0}\subseteq A_{1}\subseteq \dots \subseteq A_{N}=A$ be an increasing filtration on $A$, $M$ be 
an $A$-module with induced filtration $M_{0}\subseteq M_{1}\subseteq \dots \subseteq M_{N}=M$. 

Let $C^{\bullet}(A,M)$ and $C^{\bullet}(\gr  A, \gr M)$ be the complexes obtained by taking
duals of the respective reduced bar resolutions. That is 
$C^{n}(A,M)=\Hom_\C ((A_{+})^{\otimes n}\otimes M, \C )$
and  $C^{n}(\gr A, \gr M)=\Hom_\C ((\gr A_{+})^{\otimes n}\otimes \gr M, \C )$. 
Set ${A_{+}}_{j}=A_{j}\cap A_{+}$. 

The filtrations above induce a downward filtration on the complex $C^{\bullet}(A,M)$ as follows. Define 
\begin{gather*}
B^{n}_{[< t ]}=\sum_{ \eta+\sum\gamma_{i} < t} {A_{+}}_{\gamma_1}
\otimes {A_{+}}_{\gamma_2} \otimes \cdots \otimes
{A_{+}}_{\gamma_n} \otimes M_{\eta},\\ 
B^{n}_{[\leq t]}=\sum_{
\eta+\sum\gamma_{i}  \leq t} {A_{+}}_{\gamma_1} \otimes
{A_{+}}_{\gamma_2} \otimes\cdots \otimes {A_{+}}_{\gamma_n}\otimes M_{\eta}.
\end{gather*}
Now observe that $B^{n}_{[<t]}\subseteq B^{n}_{[\leq t]}$. Set 
$$C^{n}(A,M)_{[<t]}=\Hom_\C ((A_{+}^{\otimes n}\otimes M)/B^{n}_{[<t]},\C ),\ \ 
C^{n}(A,M)_{[\leq t]}=\Hom_\C ((A_{+}^{\otimes n}\otimes M)/B^{n}_{[\leq t]},\C );$$ then
$C^{n}(A,M)_{[\leq t]}\subseteq C^{n}(A,M)_{[< t ]}$. Moreover, if $s,t \in {\mathbb N}$ with $s<t $ then ${C^{n}(A,M)_{[< t]}}\subseteq {C^{n}(A,M)_{[< s ]}}.$

The grading on $\gr A$ leads in a natural way to a grading of the complex $C^{\bullet}(\gr  A, \gr M)$. 
Let $C^{\bullet}(\gr  A, \gr  M)_{[t]}$ denote the graded component corresponding to $t$; then we can identify
${C^\bullet(A,M)_{[<t]}}/{C^\bullet(A,M)_{[\leq t ]}}$ with ${C^{\bullet}(\gr  A, \gr (M))_{[t]}}.$ From these filtrations on the cobar complexes we obtain the following result. 

\begin{prop} \label{P:Mayspectral}  Let $A$ be a finite-dimensional augmented algebra with a multiplicative filtration $A_{0}\subseteq A_{1}\subseteq \dots \subseteq A_{N}=A$, and let 
$M$ be an $A$-module. There exists a spectral sequence 
$$E_{1}^{i,j}(M)=\operatorname{Ext}^{i+j}_{\operatorname{gr}A}(\C ,\operatorname{gr }M)_{[i]}\Rightarrow \operatorname{Ext}^{i+j}_{A}(\C ,M).$$ 
Furthermore, for each $r$, $E_{r}^{\bullet,\bullet}(M)$ is a differential graded module over the differential graded algebra $E_{r}^{\bullet, \bullet}(\C )$.  
\end{prop} 

\subsection{} In this section, let $A=u_{\zeta}({\mathfrak b})$ and $\bar{A}=\gr A$. The following theorem can be regarded as an adjointness statement 
between extension groups for $\bar{A}$ and $\Delta$. 

\begin{theorem} \label{thm:injectivemap}  Let $Q_{1}$ be a finite-dimensional module for $\bar{A}$ and $Q_{2}$ be an arbitrary $\bar{A}$-module. For $n\geq 0$ there exists 
an isomorphism 
$$\phi_{n}: \operatorname{Ext}^{n}_{\bar{A}}(Q_{1}^{*},Q_{2})\rightarrow \operatorname{Ext}^{n}_{\Delta}(\C ,Q_{1}\otimes Q_{2}).$$ 
\end{theorem} 

\begin{proof} First we need to construct the maps $\phi_{n}$. Let $\zeta$ be an element of $\Ext^{n}_{\bar{A}}(Q_{1}^{*},Q_{2})$ which, when $n>0$, 
is represented by an $n$-fold $\bar{A}$-extension: 
$$0\rightarrow Q_{2} \rightarrow T_{1} \rightarrow T_{2} \rightarrow \dots \rightarrow T_{n} \rightarrow Q_{1}^{*} \rightarrow 0.$$
By tensoring with $Q_{1}$, one gets an $n$-fold $\bar{A}\otimes \bar{A}$-extension that one can regard as a $\Delta$-extension via restriction: 
$$0\rightarrow Q_{1}\otimes Q_{2} \rightarrow Q_{1}\otimes T_{1} \rightarrow Q_{1} \otimes T_{2} \rightarrow \dots \rightarrow Q_{1}\otimes T_{n} \rightarrow Q_{1}\otimes Q_{1}^{*}\rightarrow 0.$$
This gives us a map $\widetilde{\phi}_{n}:\operatorname{Ext}^{n}_{\bar{A}}(Q_{1}^{*},Q_{2})\rightarrow \operatorname{Ext}^{n}_{\Delta}(Q_{1}\otimes Q_{1}^{*},Q_{1}\otimes Q_{2})$.   
We have a $\Delta$-map ${\mathbb C}\hookrightarrow Q_{1}\otimes Q_{1}^{*}$. This induces a map on extension groups
$$\widehat{\phi}_{n}:\operatorname{Ext}^{n}_{\Delta}(Q_{1}\otimes Q_{1}^{*},Q_{1}\otimes Q_{2})\rightarrow \operatorname{Ext}^{n}_{\Delta}(\C ,Q_{1}\otimes Q_{2}).$$ 
Now one can set $\phi_{n}=\widehat{\phi}_{n} \circ \widetilde{\phi}_{n}$. When $n=0$, $\zeta\in\Hom_{\bar A}(Q_{1}^{*},Q_{2})$, and we can define $\phi_{0}(\zeta)$ to be the composition of ${\mathbb C}\hookrightarrow Q_{1}\otimes Q_{1}^{*}$ followed by $1\otimes\zeta:Q_{1}\otimes Q_{1}^{*}\to Q_{1}\otimes Q_{2}$.

The proof that $\phi_{n}$ is injective will follow via the first five steps, below. Surjectivity will be handled with two further steps.

\vskip .5cm 
\noindent
(1) We claim that for $Q_{1}\cong \lambda$, 
$$\phi_{0}:\Hom_{\bar{A}}(\lambda^{*},Q_{2})\rightarrow \Hom_{\Delta}(\C ,\lambda\otimes Q_{2}).$$ 
is an isomorphism. 

Consider $A^{\prime}=u_{\zeta}({\mathfrak u})$ and $\bar{A}^{\prime}=\gr A^{\prime}$. Note that $\bar{A}^{\prime}$ acts trivially 
on $\lambda$. By comparing the action of $\bar{A}$ on $Q_{2}$ and $\Delta$ on $\lambda\otimes Q_{2}$, one can see that the fixed points under $A^{\prime}$ 
are the same. The result now follows by considering the weights on the set of fixed points. 

\vskip .5cm 
\noindent
(2) Next we prove that for $Q_{1}$ a finite-dimensional $\bar{A}$-module, 
$$\phi_{0}:\Hom_{\bar{A}}(Q_{1}^{*},Q_{2})\rightarrow \Hom_{\Delta}(\C ,Q_{1}\otimes Q_{2}).$$ is injective. 

Consider the following  diagram with exact rows, induced by the short exact sequences
$$0\rightarrow \lambda\rightarrow Q_{1} \rightarrow N \rightarrow 0 \qquad \text{ and } \qquad 
0\rightarrow N^{*} \rightarrow Q_{1}^{*} \rightarrow \lambda^{*} \rightarrow 0:$$ 

$$
\CD
0@>>> \Hom_{\bar{A}}(\lambda^{*},Q_{2}) @>>> \Hom_{\bar{A}}(Q_{1}^{*},Q_{2}) @>\tau_{1}>> \Hom_{\bar{A}}(N^{*},Q_{2})   \\
@.       @V\phi_{0}^{\prime}VV @V\phi_{0}VV  @V\phi_{0}^{\prime \prime}VV        \\
0@>>> \Hom_{\Delta}(\C ,\lambda\otimes Q_{2}) @>>> \Hom_{\Delta}(\C ,Q_{1}\otimes Q_{2}) @>\tau_{2}>> \Hom_{\Delta}({\mathbb C},N\otimes Q_{2})   \\
\endCD
$$
One can verify directly that the diagram commutes, and hence $\phi_{0}^{\prime \prime}(\operatorname{Im}\tau_{1})\subseteq \operatorname{Im}\tau_{2}$. Replacing the last term in row $i$ ($i=1, 2$) by $\operatorname{Im}\tau_{i}\to 0$ produces a new commutative diagram whose rows are short exact sequences. Since $\phi_{0}^{\prime}$ and $\phi_{0}^{\prime \prime}$ are monomorphisms by induction on $\dim Q_{1}$, it follows by the five lemma that 
$\phi_{0}$ is a monomorphism. 

\vskip .5cm 
\noindent 
(3) We claim that if $Q_{1}$ is injective as ${\bar{A}}$-module then 
$$\phi_{0}:\Hom_{\bar{A}}(Q_{1}^{*},Q_{2})\rightarrow \Hom_{\Delta}(\C ,Q_{1}\otimes Q_{2}).$$ 
is an isomorphism. 

From (2), we know that $\phi_{0}$ is injective. For any $f\in \Hom_{\Delta}(\C ,Q_{1}\otimes Q_{2})$, the 
image of $f$ will lie in a finite-dimensional submodule of $Q_{1}\otimes Q_{2}$. Since $Q_{1}$ is finite-dimensional and 
$Q_{2}$ is locally finite, we may assume that the image of $f$ is contained in $Q_{1}\otimes Z$ where $Z$ is a finite-dimensional submodule of $Q_{2}$. So for (3), we may assume without loss of generality that $Q_{2}$ is finite-dimensional. 

First one can reduce to the the case when $Q_{1}$ is the injective hull of $\lambda\in X_{1}$ as $\bar{A}$-module. Then 
$Q_{1}^{*}$ is the projective cover of $-\lambda$. Furthermore, for $\mu \in X_{1}$, $Q_{1}\otimes \mu$ is indecomposable and the injective hull of $\lambda\otimes \mu$. This shows that 
$\phi_{0}$ is an isomorphism in the case when $Q_{1}$ is the injective hull of $\lambda$ and $Q_{2}=\mu$. 
The functors $\Hom_{\bar{A}}(Q_{1}^{*},-)$ and $\Hom_{\Delta}(\C ,Q_{1}\otimes (-))$ are exact. One can 
apply induction on the composition length to prove the statement for arbitrary $Q_{2}$. 
\vskip .5cm 
\noindent 
(4) Let $Q_{1}$ be a finite-dimensional $\bar{A}$-module. Then 
$$\phi_{1}:\Ext_{\bar{A}}^{1}(Q_{1}^{*},Q_{2})\rightarrow \Ext^{1}_{\Delta}(\C ,Q_{1}\otimes Q_{2})$$ 
is injective. 

Consider the short exact sequences obtained where $I$ is the (finite-dimensional) injective hull of $Q_{1}$: 
\begin{equation}\label{ses1}
0\rightarrow Q_{1} \rightarrow I \rightarrow N \rightarrow 0
\qquad \text{ and } \qquad
0\rightarrow N^{*} \rightarrow I^{*} \rightarrow Q_{1}^{*} \rightarrow 0. 
\end{equation} 
These short exact sequences induce the following commutative diagram: 
$$
{\footnotesize
\minCDarrowwidth16pt
\CD
0@>>> \Hom_{\bar{A}}(Q_{1}^{*},Q_{2}) @>>> \Hom_{\bar{A}}(I^{*},Q_{2}) @>\tau_{1}>> \Hom_{\bar{A}}(N^{*},Q_{2}) @>\pi_{1}>> \Ext^{1}_{\bar{A}}(Q_{1}^{*},Q_{2}) @>>> 0  \\
@.       @V\phi_{0}VV @V\phi_{0}^{\prime}VV  @V\phi_{0}^{\prime \prime}VV   @V\phi_{1}VV    \\
0@>>> \Hom_{\Delta}(\C ,Q_{1}\otimes Q_{2}) @>>> \Hom_{\Delta}(\C,I\otimes Q_{2}) @>\tau_{2}>> \Hom_{\Delta}(\C ,N\otimes Q_{2}) @>\pi_{2}>> \Ext^{1}_{\Delta}(\C,Q_{1}\otimes Q_{2}) @>>> 0   \\
\endCD}
$$

Now according to (2) and (3), %$\phi_{0}$ and 
$\phi_{0}^{\prime \prime}$ is injective and $\phi_{0}^{\prime}$ is an isomorphism. Our goal is to show that $\phi_{1}$ is injective. 
Suppose that $\phi_{1}(x)=0$. Since $\pi_{1}$ is surjective there exists $z$ such that $\pi_{1}(z)=x$. Now by commutativity, 
$$\pi_{2}(\phi_{0}^{\prime \prime}(z))=\phi_{1}(\pi_{1}(z))=\phi_{1}(x)=0.$$ 
Therefore, $\phi_{0}^{\prime \prime}(z)\in \operatorname{ker}\pi_{2}=\operatorname{im}\tau_{2}$, and there exists 
$y$ such that $\tau_{2}(y)=\phi_{0}^{\prime \prime}(z)$. Since $\phi_{0}^{\prime}$ is an isomorphism there 
exists $w$ such that $\phi_{0}^{\prime}(w)=y$. By commutativity, 
$$\phi_{0}^{\prime \prime}(\tau_{1}(w))=\tau_{2}(\phi_{0}^{\prime}(w))=\tau_{2}(y)=\phi_{0}^{\prime \prime}(z).$$ 
Consequently, by the injectivity of $\phi_{0}^{\prime \prime}$, it follows that 
$z=\tau_{1}(w)$. Thus, $x=\pi_{1}(z)=\pi_{1}(\tau_{1}(w))=0$. 

\vskip .5cm 
\noindent 
(5) Finally, let $Q_{1}$ be a finite-dimensional $\bar{A}$-module. For $n\geq 0$, 
$$\phi_{n}:\Ext_{\bar{A}}^{n}(Q_{1}^{*},Q_{2})\rightarrow \Ext^{n}_{\Delta}(\C ,Q_{1}\otimes Q_{2})$$ 
is injective. 

For $n=0,1$, this was proved in (2) and (4). As in (4), the short exact sequences \cref{ses1} yield commutative diagrams 
for $n\geq 1$ 
$$
\CD
\Ext^{n}_{\bar{A}}(N^{*},Q_{2}) @>\epsilon_{n}>> \Ext^{n+1}_{\bar{A}}(Q_{1}^{*},Q_{2})  \\
@V\phi_{n}^{\prime \prime}VV   @V\phi_{n+1}VV    \\
\Ext^{n}_{\Delta}(\C ,N\otimes Q_{2}) @>\epsilon^{\prime}_{n}>> \Ext^{n+1}_{\Delta}(\C ,Q_{1}\otimes Q_{2})  \\
\endCD
$$
where $\epsilon_{n}$ and $\epsilon^{\prime}_{n}$ are isomorphisms. Now by induction, $\phi_{n}''$ is injective which implies that 
$\phi_{n+1}$ is injective. This completes the proof that $\phi_{n}$ is injective for all $n$. 
\vskip .5cm 
\noindent 
(6) We next show that 
$$\phi_{0}:\Hom_{\bar{A}}(Q_{1}^{*},Q_{2})\rightarrow \Hom_{\Delta}(\C ,Q_{1}\otimes Q_{2})$$ 
is an isomorphism.

We shall prove this by induction on the dimension of $Q_{1}$. If $Q_{1}$ is one-dimensional then 
(6) holds by (1). Consider the short exact sequence as in (2):
$$0\rightarrow N^{*} \rightarrow Q_{1}^{*} \rightarrow \lambda^{*} \rightarrow 0.$$ 
We have the commutative diagram with exact rows:
$$
{\footnotesize
\minCDarrowwidth22pt
\CD
0@>>> \Hom_{\bar{A}}(\lambda^{*},Q_{2}) @>>> \Hom_{\bar{A}}(Q_{1}^{*},Q_{2}) @>\tau_{1}>> \Hom_{\bar{A}}(N^{*},Q_{2}) @>\sigma_{1}>> \Ext^{1}_{\bar{A}}(\lambda^{*},Q_{2})    \\
@.       @V\phi_{0}^{\prime}VV @V\phi_{0}VV  @V\phi_{0}^{\prime \prime}VV      @V\phi_{1}^{\prime}VV  \\
0@>>> \Hom_{\Delta}(\C ,\lambda\otimes Q_{2}) @>>> \Hom_{\Delta}(\C ,Q_{1}\otimes Q_{2}) @>\tau_{2}>> \Hom_{\Delta}(\C ,N\otimes Q_{2}) @>\sigma_{2}>> \Ext^{1}_{\Delta}(\C,\lambda\otimes Q_{2})  
\endCD}
$$
Here $\phi_{0}^{\prime}$, $\phi_{0}^{\prime \prime}$ are isomorphisms and $\phi_{0}$, $\phi_{1}^{\prime}$ are injective. The  diagram above induces the following commutative diagram with exact rows: 
$$
{\footnotesize
\CD
0@>>> \Hom_{\bar{A}}(Q_{1}^{*},Q_{2})/ \Hom_{\bar{A}}(\lambda^{*},Q_{2})  @>\tau_{1}>> \Hom_{\bar{A}}(N^{*},Q_{2}) @>\sigma_{1}>> \operatorname{Im}\sigma_{1} @>>> 0   \\
@.       @V\phi_{0}VV @V\phi_{0}^{\prime \prime}VV     @V\phi_{1}^{\prime}VV  \\
0@>>> \Hom_{\Delta}(\C ,Q_{1}\otimes Q_{2})/\Hom_{\Delta}(\C ,\lambda\otimes Q_{2}) @>\tau_{2}>> \Hom_{\Delta}(\C ,N\otimes Q_{2}) @>\sigma_{2}>> \operatorname{Im}\sigma_{2} @>>>0   \\
\endCD}
$$
Here $\phi_{0}$ and $\phi_{1}^{\prime}$ are injective and $\phi_{0}^{\prime \prime}$ is an isomorphism. We claim that $\phi_{0}$ 
is surjective. Suppose that $y\in \Hom_{\Delta}(\C ,Q_{1}\otimes Q_{2})/\Hom_{\Delta}(\C,\lambda\otimes Q_{2})$, but 
$y\notin \operatorname{Im}\phi_{0}$. Then $(\phi_{0}^{\prime \prime})^{-1}\tau_{2}(y)\notin \operatorname{Im}\tau_{1}$. Since $\operatorname{ker}\sigma_{1}=\operatorname{Im}\tau_{1}$, it 
follows that $\sigma_{1}(\phi_{0}^{\prime \prime})^{-1} \tau_{2}(y)\neq 0$. From the injectivity of $\phi_{1}^{\prime}$, one has 
$\phi_{1}^{\prime}\sigma_{1}(\phi_{0}^{\prime \prime})^{-1}\tau_{2}(y)\neq 0$, and by commutativity, 
$$0\neq \sigma_{2}\phi_{0}^{\prime \prime}(\phi_{0}^{\prime \prime})^{-1} \tau_{2}(y)=\sigma_{2}\tau_{2}(y).$$ 
This is a contradiction because $\sigma_{2}\tau_{2}=0$. 

Combining this with (2) shows that $\phi_{0}$ must be an isomorphism. 
\vskip .5cm 
\noindent 
(7) We show that 
$$\phi_{n}:\Ext^{n}_{\bar{A}}(Q_{1}^{*},Q_{2})\rightarrow \Ext^{n}_{\Delta}(\C ,Q_{1}\otimes Q_{2})$$ 
is an isomorphism. 

Consider the short exact sequence 
$$0\rightarrow Q_{2} \rightarrow I \rightarrow N \rightarrow 0$$
where $I$ is injective. Then one has a commutative diagram with exact rows
$$
\CD
%0@>>> \Hom_{\bar{A}}(Q_{1}^{*},Q_{2}) @>>> \Hom_{\bar{A}}(Q_{1}^{*},I) @>\epsilon_{1}>> 
\Hom_{\bar{A}}(Q_{1}^{*},N) @>>> \Ext^{1}_{\bar{A}}(Q_{1}^{*},Q_{2}) @>>>0    \\
%@.       @V\phi_{0}VV @V\phi_{0}^{\prime}VV  
@V\phi_{0}^{\prime \prime}VV      @V\phi_{1}VV  \\
%0@>>> \Hom_{\Delta}(\C ,Q_{1}\otimes Q_{2}) @>>> \Hom_{\Delta}(\C ,Q_{1}\otimes I) @>\epsilon_{2}>> 
\Hom_{\Delta}(\C ,Q_{1}\otimes N) @>>> \Ext^{1}_{\Delta}(\C ,Q_{1}\otimes Q_{2}) @>>>0  \\
\endCD
$$
Here, %$\phi_{0}$, $\phi_{0}^{\prime}$ and 
$\phi_{0}^{\prime \prime}$ is an isomorphism by (6) and $\phi_{1}$ is injective by (5). It follows immediately that
%Since $\phi_{0}$ and $\phi_{0}^{\prime}$ are isomorphisms, 
%$\operatorname{Im}\epsilon_{1}=\operatorname{Im}\epsilon_{2}$. Consequently, by the  five lemma, 
$\phi_{1}$ is surjective, and hence is an isomorphism. One can use a dimension shifting argument, as in (5), to show that $\phi_{n}$ is an isomorphism for $n\geq 0$. The proof is now complete. 
\end{proof}

%%%%%%%%%
%%References
%%%%%%%%%

\let\section=\oldsection
\bibliographystyle{eprintamsmath}
\bibliography{quantum}

\iffalse
\providecommand{\bysame}{\leavevmode\hbox
to3em{\hrulefill}\thinspace}

\fi

\end{document}